\documentclass[a4paper,11pt,twoside]{article}

\author{
\normalsize William Salkeld \\[8pt]
         \small  Laboratoire J.A.Dieudonn\'e \\ 
         \small  Universit\'e de Nice Sophia-Antipolis \\
         \small  Parc Valrose
         \small France-06108 NICE Cedex 2 \\
        \small  salkeld@unice.fr 
}

\date{ \currenttime, \ddmmyyyydate\today}

\usepackage[utf8]{inputenc}
\usepackage[T1]{fontenc}
\usepackage[english]{babel}

\usepackage{charter}

\usepackage{indentfirst}
\usepackage{xcolor}
\usepackage{verbatim}
\usepackage{hyperref}
\usepackage[top=3.cm, bottom=4.0cm, left=2.2cm, right=2.2cm]{geometry}
\usepackage{amsmath}
\usepackage{amsthm}	
\usepackage{amsfonts}	
\usepackage{amssymb, bbm}
\usepackage{shuffle}

\usepackage{pifont}
\usepackage{enumerate}
\usepackage[nobysame
           ,alphabetic
           ,initials
           ]{amsrefs}

\usepackage{graphicx}
\usepackage{mathrsfs}

\usepackage{
						ulem		
} 
\numberwithin{equation}{section}
 \usepackage[nodayofweek]{datetime}

\theoremstyle{plain}
\newtheorem{theorem}{Theorem}[section]
\newtheorem{lemma}[theorem]{Lemma}
\newtheorem{proposition}[theorem]{Proposition}
\newtheorem{corollary}[theorem]{Corollary}

\newtheorem{definition}[theorem]{Definition}

\newtheorem{remark}[theorem]{Remark}
\newtheorem{example}[theorem]{Example}
\newtheorem{assumption}[theorem]{Assumption}

\newcommand{\bB}{\mathbb{B}}

\newcommand{\bE}{\mathbb{E}}

\newcommand{\bN}{\mathbb{N}}
\newcommand{\bP}{\mathbb{P}}

\newcommand{\bR}{\mathbb{R}}

\newcommand{\bW}{\mathbb{W}}

\newcommand{\bZ}{\mathbb{Z}}

\newcommand{\cA}{\mathcal{A}}
\newcommand{\cB}{\mathcal{B}}

\newcommand{\cE}{\mathcal{E}}
\newcommand{\cF}{\mathcal{F}}

\newcommand{\cH}{\mathcal{H}}

\newcommand{\cK}{\mathcal{K}}
\newcommand{\cL}{\mathcal{L}}
\newcommand{\cM}{\mathcal{M}}

\newcommand{\cP}{\mathcal{P}}

\newcommand{\cR}{\mathcal{R}}
\newcommand{\cS}{\mathcal{S}}

\newcommand{\fB}{\mathfrak{B}}
\newcommand{\fC}{\mathfrak{C}}

\newcommand{\fE}{\mathfrak{E}}
\newcommand{\fF}{\mathfrak{F}}
\newcommand{\fG}{\mathfrak{G}}
\newcommand{\fH}{\mathfrak{H}}

\newcommand{\fL}{\mathfrak{L}}
\newcommand{\fM}{\mathfrak{M}}
\newcommand{\fN}{\mathfrak{N}}

\newcommand{\fQ}{\mathfrak{Q}}

\newcommand{\fS}{\mathfrak{S}}

\newcommand{\fc}{\mathfrak{c}}
\newcommand{\fl}{\mathfrak{l}}

\newcommand{\fs}{\mathfrak{s}}

\newcommand{\rk}{\mathbf{K}}
\newcommand{\rh}{\mathbf{h}}
\newcommand{\rw}{\mathbf{W}}
\newcommand{\rx}{\mathbf{X}}
\newcommand{\ry}{\mathbf{Y}}
\newcommand{\rId}{\mathbf{1}}

\newcommand{\rS}{\mathbf{S}}

\DeclareMathOperator{\supp}{supp}

\DeclareMathOperator{\erf}{erf}

\newcommand{\1}{\mathbbm{1}}

\title{Small ball probabilities, metric entropy \\and Gaussian rough paths}

\begin{document}

\maketitle

\renewcommand*{\thefootnote}{\arabic{footnote}}

\begin{abstract} 
We study the Small Ball Probabilities (SBPs) of Gaussian rough paths. While many works on rough paths study the Large Deviations Principles (LDPs) for stochastic processes driven by Gaussian rough paths, it is a noticeable gap in the literature that SBPs have not been extended to the rough path framework. 

LDPs provide macroscopic information about a measure for establishing Integrability type properties. SBPs provide microscopic information and are used to establish a locally accurate approximation for a measure. Given the compactness of a Reproducing Kernel Hilbert space (RKHS) ball, its metric entropy provides invaluable information on how to approximate the law of a Gaussian rough path. 

As an application, we are able to find upper and lower bounds for the rate of convergence of an empirical rough Gaussian measure to its true law in pathspace. 
\end{abstract} 

\vspace{0.3cm}

\noindent


\setcounter{tocdepth}{1}
\tableofcontents

\section{Introduction}

Small Ball Probabilites (SBPs), sometimes referred to as small deviations principles, study the asymptotic behavour of the measure of a ball of radius $\varepsilon\to 0$. Given a measure $\cL$ on a metric space $(E, d)$ with Borel $\sigma$-algebra $\cB$, we refer to the SBP around a point $x_0$ as 
$$
\log\bigg(\cL\Big[ \big\{ x\in E: d(x, x_0)<\varepsilon\big\}\Big]\bigg) \qquad \varepsilon \to 0. 
$$
This is in contrast to a Large Deviations Principle (LDP) which considers the asymptotic behaviour for the quantity
$$
\log\bigg(\cL\Big[ \big\{ x\in E: d(x, x_0)>a \big\}\Big]\bigg) \qquad a\to \infty. 
$$
LDPs have proved to be a powerful tool for quantifying the tails of Gaussian probability distributions that have been sucessfully explored and documented in recent years, see for example \cites{bogachev1998gaussian, ledoux2013probability} and references therein. Similar results have been extended to a wide class of probability distributions, see for example \cites{varadhan1984large, DemboZeitouni2010}. However, the complexity of SBPs has meant there has been a generally slower growth in the literature. This is not to detract from their usefulness: there are many insightful and practical applications of SBPs to known problems, in particular the study of compact operators, computation of Hausdorff dimension and the rate of convergence of empirical and quantized distributions.

As a motivational example, let $\cL$ be a Gaussian measure on $\bR^d$ with mean $0$ and identity covariance matrix. Then
$$
\cL\Big[ \big\{ x\in \bR^d: |x|_2<\varepsilon\big\} \Big] = \frac{\Gamma(d/2) - \Gamma(d/2, \tfrac{\varepsilon^2}{2})}{\Gamma(d/2)} \sim \frac{2\cdot \varepsilon^d}{\Gamma(d+1) \cdot 2^{d/2} } \qquad \varepsilon \to 0. 
$$
Therefore an application of l'H\^opital's rule yields
$$
\fB_{0, 2}(\varepsilon)=-\log\bigg( \cL\Big[ \big\{ x\in \bR^d: |x|_2<\varepsilon\big\} \Big]\bigg) \sim d \cdot \log(\varepsilon^{-1}) \qquad \varepsilon \to 0. 
$$
Alternatively, using a different norm we have
$$
\cL\Big[ \big\{ x\in \bR^d: |x|_\infty<\varepsilon\big\} \Big] = \mbox{erf}\Big( \tfrac{\varepsilon}{\sqrt{2}}\Big)^d \sim \varepsilon^d \Big( \tfrac{2}{\pi}\Big)^{d/2} \qquad \varepsilon \to 0
$$
and we get
$$
\fB_{0, \infty}(\varepsilon)=-\log\bigg( \cL\Big[ \big\{ x\in \bR^d: |x|_\infty<\varepsilon\big\} \Big]\bigg) \sim d \cdot \log(\varepsilon^{-1}) \qquad \varepsilon \to 0. 
$$
We can think of the SBPs as capturing the Lebesgue measure of a compact set (in this case a unit ball with different norms) in the support of the measure. The question then arises, what happens as the dimensions of the domain of the Gaussian measure are taken to infinity (so that there is no Lebesgue measure to compare with) and we study Gaussian measures on Banach spaces? Similarly, how does enhancing these paths to rough paths affect their properties?

\subsubsection*{Small Ball Probabilities}

Small ball probabilities encode the shape of the cumulative distribution function for a norm around 0. For a self-contained introduction to the theory of SBPs and Gaussian inequalities, see \cite{li2001gaussian}. 

SBPs for a Brownian motion with respect to the H\"older norm were first studied in \cite{baldi1992some}. Using the Cielsielski representation of Brownian motion, the authors are able to exploit the orthogonality of the Schauder wavelets in the Reproducing Kernel Hlibert Space (RKHS) to represent the probability as a product of probabilities of 1 dimensional normal random variables. Standard analytic estimations of the Gauss Error function provide an upper and lower bound for the probability and an expression for the limit for the probability as $\varepsilon \to 0$. 

Later, the same results were extended to a large class of Gaussian processes under different assumptions for the covariance and different choices of Banach space norms, see for example \cites{kuelbs1993small,kuelbs1995small, stolz1996some} and others. 

In \cite{dobbs2014small}, the author studies some SBPs for Levy Area of Brownian motion by treating it as a time-changed Brownian motion. However, there are no works studying SBPs for rough paths. 

The metric entropy of a set is a way of measuring the ``Compactness'' of a compact set. For a neat introduction to the study of entropy and some of its applications, see \cite{carl1990entropy} and \cite{edmunds2008function}. The link between SBPs for Gaussian measures on Banach spaces and metric entropy is explored in \cite{Kuelbs1993} and later extended in \cite{li1999approximation} to encompass the truncation of Gaussian measures. SBP results for integrated Brownian motion, see \cite{gao2003integrated}, were used to compute the metric entropy of $k$-monotone functions in \cite{gao2008entropy}. The link between the entropy of the convex hull of a set and the associated Gaussian measure is explored in \cites{gao2004entropy, kley2013kuelbs}. For a recent survey on Gaussian measures and metric entropy, see \cite{kuhn2019gaussian}. 

There is a natural link between the metric entropy of the unit ball of the RKHS of a Gaussian measure and the quantization problem. Using the LDPs of the Gaussian measure, one can easily find a ball (in the RKHS) with measure $1-\varepsilon$ where $0<\varepsilon \ll 1$. Given the $\varepsilon$ entropy of this set, the centres of the minimal cover represent a very reasonable ``guess'' for an optimal quantization since the Gaussian measure conditioned on the closure of this set is ``close'' to uniform. For more details, see \cites{graf2003functional, dereich2003link}. Sharp estimates for Kolmogorov numbers, an equivalent measure to metric entropy, are demonstrated in \cite{luschgy2004sharp}. 

More recently, SBPs have been applied to Baysian inference and machine learning, see for example \cites{van2007bayesian, vaart2011information, aurzada2009small2}. 

\subsubsection*{Gaussian correlation inequalities}

A key step in the proof of many SBP results is the use of a correlation inequality to lower or upper bound a probability of the intersection of many sets by a product of the probabilities of each set. Thus a challenging probability computation can be simplified by optimising over the choice of correlation strategically. 

The Gaussian correlation inequality states that for any two symmetric convex sets $A$ and $B$ in a separable Banach space and for any centred Gaussian measure $\cL$ on $E$, 
$$
\cL[A\cap B] \geq \cL[A] \cL[B]. 
$$
The first work which considers a special case of this result was conjectured in \cite{dunnett1955approximations}, while the first formal statement was made in \cite{gupta1972inequalities}. 

While the inequality remained unproven until recently, prominent works proving special examples and weaker statements include \cites{khatri1967certain,sidak1968multivariate} (who independently proved the so called \v{S}id\'ak's Lemma), \cite{pitt1977gaussian} and \cite{li1999gaussian}. The conjecture was proved in 2014 by Thomas Royen in a relatively uncirculated ArXiv submission \cite{royen2014simple} and did not come to wider scientific attention for another three years in \cite{latala2017royen}. 

Put simply, the idea is to minimise a probability for a collection of normally distributed random variables by varying the correlation. Applications of these inequalities are wide ranging and vital to the theory of Baysian inference. 

\subsubsection*{Rough paths and enhanced Gaussian measures}

Since their inception in \cite{lyons1998differential}, rough paths have proved a powerful tool in understanding stochastic processes. In a nut shell, the theory states that given an irregular white noise propagating a differential equation, one is required to know the path and the iterated integrals of the noise for a rich class of solutions. This path taking values on the characters of a Hopf algebra and is referred to as the signature. 

An important step in the development of the theory of rough paths was the work of \cite{ledoux2002large} which studies the LDPs of an enhanced Brownian motion, the so-called lift of the path of a Brownian motion to its signature. The authors prove a Large Deviations Principle and a support theorem for the law of the enhanced Brownian motion as a measure over the collection of rough paths with respect to the rough path metric. Then, by the continuity of the It\^o-Lyons map the LDP can be extended to the solution of any rough differential equation driven by the enhanced Brownian motion. 

Originally, rough paths were used to give a pathwise meaning to the solutions of stochastic differential equations where previously only a probabilistic meaning was known. However, there are an increasing number of works that study measures on the collection of rough paths motivated by the study of systems of interacting particles. 

In general, the study of measures over rough paths has been focused on the macroscopic properties. This was natural given the signature contains more information than the path on its own and it is not immediately clear that this extra information does not render the objects non-integrable. Questions of integrability of rough paths were addressed in \cites{friz2010generalized,cass2013integrability}. These were used to study rough differential equations that depend on their own distribution, the so called McKean-Vlasov equations in \cite{CassLyonsEvolving}. More recently, there has been a rapid expansion of this theory, see \cites{coghi2018pathwise,  2019arXiv180205882.2B, cass2019support}. Of particular interest to this work is \cite{deuschel2017enhanced} which studies the convergence of the empirical measure obtained by sampling $n$ enhanced Brownian motions to the law of an enhanced Brownian motion. 

The author was unable to find material in the literature pertaining to the microscopic properties of distributions over the collection of rough paths. This work came out of a need to better understand interacting particle systems driven by Gaussian noises, although we emphasise that no results in this paper need be restricted to that framework. 

\subsubsection*{Our contributions}

The structure of this paper is as follows: Firstly, we introduce necessary material and notations in Section \ref{section:Prelim}. In order to extend the theory of Gaussian measures on Banach spaces to the framework of rough paths, we need to rephrase several well known Gaussian inequalities and prove new correlation inequalites. This is done in Section \ref{section:EnGaussInequal}. While technical, these results are stronger than we require and represent an extension of the theory of correlation inequalities to elements of the Wiener It\^o chaos expansion. 

The main contribution of this work is the computation of SBPs for Gaussian rough paths with the rough path H\"older metric. These results are solved in Section \ref{section:SmallBallProbab}. We remark that the discretisation of the H\"older norm in Lemma \ref{lemma:DiscretisationNorm} was unknown to the author and may be of independent interest for future works on rough paths. 

Finally, Sections \ref{section:MetricEntropy} and \ref{section:OptimalQuant} are applications of Theorem \ref{Thm:SmallBallProbab} following known methods that are adapted to the rough path setting. Of particular interest are Theorems \ref{thm:QuantizationRateCon} and \ref{thm:EmpiricalRateCon1} 
which provide an upper and lower bound for the rate of convergence for the empirical rough Gaussian measure. 

\section{Preliminaries}
\label{section:Prelim}

We denote by $\bN=\{1,2,\cdots\}$ the set of natural numbers and $\bN_0=\bN\cup \{0\}$, $\bZ$ and $\bR$ denote the set of integers and real numbers respectively. $\bR^+=[0,\infty)$. By $\lfloor x \rfloor$ we denote the largest integer less than or equal to $x\in \bR$. $\1_A$ denotes the usual indicator function over some set $A$. Let $e_j$ be the unit vector of $\bR^d$ in the $j^{th}$ component and $e_{i, j} = e_i \otimes e_j$ be the unit vector of $\bR^d \otimes \bR^d$. 

For sequences $(f_n)_{n\in \bN}$ and $(g_n)_{n\in\bN}$, we denote 
\begin{align*}
f_n \lesssim g_n \ \ \iff  \ \ \limsup_{n\to \infty} \frac{f_n}{g_n}\leq C, 
\qquad \textrm{and}\qquad 
f_n \gtrsim g_n \ \ \iff  \ \ \liminf_{n\to \infty} \frac{f_n}{g_n}\geq C. 
\end{align*}
where $C$ is a positive constant independent of the limiting variable. When $f_n \lesssim g_n$ and $f_n \gtrsim g_n$, we say $f_n \approx g_n$. This is distinct from
\begin{align*}
f_n \sim g_n \ \ \iff \ \ \lim_{n\to \infty} \frac{f_n}{g_n} = 1. 
\end{align*}

We say that a function $L:(0,\infty) \to (0, \infty)$ is \emph{slowly varying at infinity} if $\forall s>0$
$$
\lim_{t\to \infty} \frac{L(st)}{L(t)} = 1. 
$$
A function $x \mapsto \phi(1/x)$ is called \emph{regularly varying at infinity} with index $a>0$ if there exists a function $L$ which is slowly varying at infinity such that
$$
\phi(\varepsilon) = \varepsilon^{-a} L \big(\tfrac{1}{\varepsilon}\big). 
$$

\subsection{Gaussian Theory}

\begin{definition}
Let $E$ be a separable Banach space equipped with its cylinder $\sigma$-algebra $\cB$. A Gaussian measure $\cL$ is a Borel probability measure on $(E, \cB)$ such that the pushforward measure of each element of the dual space $E^*$ is a Gaussian random variable. Thus the measure $\cL$ is uniquely determined in terms of the covariance bilinear form $\cR:E^* \times E^* \to \bR$ by
$$
\cR[f, g]:= \int_E f(x) \cdot g(x) d\cL(x). 
$$
The covariance Kernel, $\cS:E^* \to E$ is defined in terms of the Pettis integral
$$
\cS[f]:= \int_E x \cdot f(x) d\cL(x). 
$$
Denote by $\cH$ the Hilbert space obtained by taking the closure of $E^*$ with respect to the inner product induced by the form $\cR$. The covariance kernel has spectral representation $\cS= i i^*$ where $i$ is the compact embedding of $\cH$ into $E$. 

We refer to $\cH$ as the Reproducing Kernel Hilbert Space (RKHS) of Gaussian measure $\cL$. 
When the Hilbert space $\cH$ is dense in the Banach space $E$, the triple $(E, \cH, i)$ is called an Abstract Wiener space. 
\end{definition}

We denote the unit ball in the RKHS norm as $\cK$. It is well known that the set $\cK$ is compact in the Banach space topology. 

\begin{proposition}[Borell's Inequality]
\label{pro:BorellInequal}
Let $\Phi(x):=\int_{-\infty}^x \frac{1}{\sqrt{2\pi}} \exp( -y^2/2) dy$. Let $\cL$ be a Gaussian measure on a separable Banach space $E$. Let $\cK:=\{h\in \cH: \|h\|_\cH\leq1\}$ and let $A$ be a Borel subset of $E$. Then
$$
\cL_*( A + t\cK ) \geq \Phi\Big( t + \Phi^{-1}( \cL(A))\Big)
$$
where $\cL_*$ is the inner measure of $\cL$ and is chosen to avoid measurability issues with the set $A + t\cK$. 
\end{proposition}

The proof can be found in \cite{ledoux1996isoperimetry}. 

\subsection{Rough paths}
\label{subsec:RoughPaths}

Throughout this paper, we will use the notation for increments of a path $X_{s, t} = X_t - X_s$ for $s\leq t$. Rough paths were first introduced in \cite{lyons1998differential}. 
We will be most closely following \cite{frizhairer2014}. 

\subsubsection{Algebraic material}

We denote $T^{(2)}(\bR^d)$ to be the tensor space $\bR \oplus \bR^d \oplus ( \bR^d \otimes \bR^d)$. This has a natural vector space structure along with a non-commutative product $\boxtimes$ with unit $(1, 0,0)$ defined by 
$$
(a, b, c) \boxtimes (a', b'. c') = (a\cdot a', a\cdot b' + a' \cdot b, a\cdot c' + a'\cdot c + b \otimes b').
$$ 
For $i=0,1,2$, we denote the canonical projection $\pi_i: T^{(2)}(\bR^d) \to (\bR^d)^{\otimes i}$. 

The subset $G^{(2)}(\bR^d) = \{ (a, b, c) \in T^{(2)}(\bR^d): a=1\}$ forms a non-commutative Lie group with group operation $\boxtimes$ and inverse $(1, b, c)^{-1} = (1, -b, -c + b\otimes b)$. This turns out to be the step-2 nilpotent Lie group with $d$ generators. 

The subset $\fL^{(2)} (\bR^d) = \{ (a, b, c) \in T^{(2)}(\bR^d): a=0\}$ forms a Lie algebra with the Lie brackets 
$$
[\fl_1, \fl_2]_{\boxtimes} = \fl_1 \boxtimes \fl_2 - \fl_2 \boxtimes \fl_1. 
$$
There exist bijective diffeomorphisms between $\fL^{(2)}(\bR^d)$ and $G^{(2)}(\bR^d)$ called the exponential map $\exp_\boxtimes :\fL^{(2)}(\bR^d) \to G^{(2)}(\bR^d)$ and logarithm map $\log_{\boxtimes}: G^{(2)}(\bR^d) \to \fL^{(2)}(\bR^d)$ defined by
$$
\exp_\boxtimes\Big( (0,b,c) \Big) =  (1, b, c +\tfrac{1}{2} b\otimes b), 
\quad 
\log_\boxtimes\Big( (1, b, c) \Big) = (0, b, c - \tfrac{1}{2} b\otimes b). 
$$

We define the dilation $(\delta_t)_{r>0}$ on the Lie algebra $\fL^{(2)}(\bR^d)$ to be the collection of automorphisms of $\fL^{(2)}(\bR^d)$ such that $\delta_s \delta_t = \delta_{st}$. The dilation can also be extended to the Lie group by considering $\delta_t[\log_\boxtimes]$. A homogenous group is any Lie group whose Lie algebra is endowed with a family of dilations. 

A homogeneous norm on a homogeneous group $G$ is a continuous function $\|\cdot\|:G \to \bR^+$ such that $\| g\|=0 \iff g=\rId$ and $\| \delta_t[g]\| = |t|\cdot \| g\|$. A homogeneous norm is called subadditive if $\| g_1\boxtimes g_2\|\leq \| g_1\| + \|g_2\|$ and called symmetric if $\|g^{-1}\| = \|g\|$. 

When a homogenous norm is subadditive and symmetric, it induces a left invariant metric on the group called the Carnot-Caratheodory metric which we denote $d_{cc}$. We will often write
$$
\| \rx \|_{cc} = d_{cc}( \rId, \rx)
$$
All homogeneous norms are equivalent, so we will often shift between homogeneous norms that are most suitable for a given situation. 

Examples of a homogeneous norm include
\begin{equation}
\label{eq:HomoNorm}
\| g \|_{G^{(2)}} = \sum_{A \in \cA_2} \Big| \big\langle \log_\boxtimes (g), e_A \big\rangle \Big|^{1/|A|}, 
\qquad 
\| g \|_{G^{(2)}} = \sup_{A \in \cA_2} \Big| \big\langle \log_\boxtimes (g), e_A \big\rangle \Big|^{1/|A|}. 
\end{equation}
where $\{ e_A: A\in \cA_2\}$ is a basis of the vector space $T^{(2)}(\bR^d)$. 

\subsubsection{Geometric and weak geometric rough paths}

\begin{definition}
Let $\alpha\in (\tfrac{1}{3}, \tfrac{1}{2}]$. A path $\rx:[0,T] \to G^2(\bR^d)$ is called an $\alpha$-rough paths if for all words $A \in \cA_2$, 
\begin{align}
\label{eq:def:rough-paths}
\rx_{s, u} = \rx_{s, t} \boxtimes \rx_{t, u}
\quad \mbox{and} \quad
\sup_{A \in \cA_2} \sup_{s, t\in [0,T]}& \frac{\langle \rx_{s, t}, e_A \rangle }{|t-s|^{\alpha |A|}}< \infty. 
\end{align}
In, in addition, we have that $\rx$ satisfies
\begin{equation}
\label{eq:def:rough-paths2}
\mbox{Sym}\Big( \pi_2 \big[ \rx_{s, t} \big] \Big) = \tfrac{1}{2} \pi_1\big[ \rx_{s,t}\big] \otimes \pi_1\big[ \rx_{s,t} \big], 
\end{equation}
then we say $\rx$ is a weakly geometric rough path. The set of weakly geometric rough paths is denoted $WG\Omega_\alpha(\bR^d)$. 
\end{definition}

The first relation of Equation \eqref{eq:def:rough-paths} is often called the algebraic \emph{Chen's relation} and the second is referred to as the analytic \emph{regularity condition}. 

\begin{definition}
For $\alpha$-rough paths $\rx$ and $\ry$, we denote the $\alpha$-H\"older rough path metric to be
\begin{equation}
\label{eq:HolderDef}
d_\alpha( \rx, \ry) = \| \rx^{-1} \boxtimes \ry\|_\alpha = \sup_{s, t\in[0,T]} \frac{\Big\| \rx_{s,t}^{-1}\boxtimes \ry_{s,t} \Big\|_{cc} }{|t-s|^\alpha}. 
\end{equation}

By quotienting with respect to $\rx_0$, one can make this a norm. We use the convention that $\| \rx \|_{\alpha} = d_\alpha(\rId, \rx )$. 
\end{definition}

\begin{definition}
For a path $x\in C^{1-var}([0,T]; \bR^d)$, the iterated integrals of $x$ are canonically defined using Young integration. The collection of iterated integrals of the path $x$ is called the truncated signature of $x$ and is defined as
$$
S(x)_{s, t}:= \rId + \sum_{n=1}^\infty \int_{s\leq u_1\leq ... \leq u_n\leq t} dx_{u_1} \otimes ... \otimes dx_{u_n} \in T(\bR^d). 
$$
In the same way, the truncated Signature defined by its increments
$$
S_2(x)_{s, t}:= \rId + x_{s, t} + \int_{s\leq u_1\leq u_2\leq t} dx_{u_1} \otimes dx_{u_2} \in T^{(2)}(\bR^d). 
$$
The closure of the set $\{ S_2(x): x \in C^{1-var}([0,T], \bR^d)\}$ with respect to the $\alpha$-H\"older rough path metric is the collection of \emph{geometric rough paths} which we denote by $G\Omega_\alpha(\bR^d)$. 

It is well known that $G\Omega_\alpha(\bR^d) \subsetneq WG\Omega_\alpha(\bR^d)$. 
\end{definition}

\subsubsection{The translation of rough paths}

We define the map $\#: G^2(\bR^d \oplus \bR^d) \to G^2(\bR^d)$ to be the unique homomorphism such that for $v_1, v_2\in \bR^d$, $\#[ \exp_\boxtimes(v_1 \oplus v_2)] = \exp_\boxtimes(v_1+v_2)$. 

\begin{definition}
\label{dfn:TranslatedRoughPath}
Let $\alpha, \beta>0$ such that $\alpha+\beta>1$ and $\alpha \in (\tfrac{1}{3}, \tfrac{1}{2}]$. Let $(\rx, h)\in C^{\alpha}([0,T]; G^{2}(\bR^d)) \times C^{\beta}([0,T]; \bR^d)$. We define the \emph{translation of the rough path} $\rx$ by the path $h$, denoted $T^h(\rx)\in  C^{\alpha}([0,T]; G^{2}(\bR^d))$ to be
$$
T^h(\rx) = \#\Big[ S_{2}(\rx\oplus h)\Big]
$$
\end{definition} 

\begin{lemma}
\label{lemma:RPTranslation1}
Let $\alpha\in ( \tfrac{1}{3}, \tfrac{1}{2}]$ and let $\rx\in WG\Omega_{\alpha}(\bR^d)$. Let $p=\tfrac{1}{\alpha}$ and let $q>1$ such that $1= \tfrac{1}{p} + \tfrac{1}{q}$. Let $h:[0,T] \to \bR^d$ satisfy that
$$
\| h\|_{q, \alpha, [0,T]}:= \sup_{t, s\in[0,T]} \frac{ \| h\|_{q-var, [s, t]}}{|t-s|^{\alpha}} <\infty. 
$$
Then $T^h( \rx)\in WG\Omega_\alpha(\bR^d)$ and there exists $C=C(p, q)>0$ such that
\begin{equation}
\label{eq:RPTranslation1}
\| T^h( \rx)\|_\alpha \leq C \Big( \| \rx\|_\alpha + \| h\|_{q, \alpha, [0,T]} \Big)
\end{equation}
\end{lemma}
The proof of Lemma \ref{lemma:RPTranslation1} is a simple adaption of the ideas found in the proof of \cite{cass2013integrability}*{Lemma 3.1} and was originally set as an exercise in \cite{frizhairer2014}. 

\begin{proof}
By construction from Definition \ref{dfn:TranslatedRoughPath}, we have
\begin{align}
\nonumber
\Big\langle T^h(\rx)_{s, t}, e_{i, j} \Big\rangle =& \Big\langle \rx_{s, t}, e_{i, j} \Big\rangle
+ \int_s^t \big\langle X_{s, r}, e_i \big\rangle \cdot d\big\langle h_r, e_j\big\rangle
+ \int_s^t \big\langle h_{s, r}, e_i \big\rangle \cdot d\big\langle X_r, e_j\big\rangle
\\
\label{eq:lemma:RPTranslation1-pf1.1}
&+ \int_s^t \big\langle h_{s, r}, e_i \big\rangle \cdot d\big\langle h_r, e_j\big\rangle. 
\end{align}
Each integral of Equation \eqref{eq:lemma:RPTranslation1-pf1.1} can be defined using Young integration, so that we have 
\begin{align*}
\int_s^t \big\langle X_{s, r}, e_i \big\rangle \cdot d\big\langle h_r, e_j\big\rangle \leq& C \| X \|_{p-var, [s,t]} \cdot \| h \|_{q-var, [s, t]}
\\
\int_s^t \big\langle h_{s, r}, e_i \big\rangle \cdot d\big\langle X_r, e_j\big\rangle \leq& C \| X \|_{p-var, [s,t]} \cdot \| h \|_{q-var, [s, t]}
\\
\int_s^t \big\langle h_{s, r}, e_i \big\rangle \cdot d\big\langle h_r, e_j\big\rangle \leq& C \| h \|_{q-var, [s,t]} \cdot \| h \|_{q-var, [s, t]}
\end{align*}
for some uniform constant $C>0$ dependent on $p$ and $q$. Combining these with an equivalent homogeneous norm \eqref{eq:HomoNorm} implies \eqref{eq:RPTranslation1}. 

Finally, to verify Equation \eqref{eq:def:rough-paths2}, we recall that the Young integrals satisfy an integration by parts formula so that
\begin{align*}
\mbox{Sym}\Big( \pi_2[T^h(\rx)_{s, t}] \Big) =& \mbox{Sym}\Big( \pi_2[\rx_{s, t}] \Big) + \mbox{Sym}\Big( \int_s^t X_{s, r} \otimes dh_r \Big) + \mbox{Sym}\Big( \int_s^t h_{s, r} \otimes dX_r \Big) 
\\
&+ \mbox{Sym}\Big( \int_s^t h_{s, r} \otimes dh_r \Big)
\\
=&\tfrac{1}{2} \bigg( \pi_1\big[ \rx_{s,t}\big] \otimes \pi_1\big[ \rx_{s,t} \big] + X_{s, t} \otimes h_{s, t} + h_{s, t} \otimes X_{s, t} + h_{s, t} \otimes h_{s, t} \bigg)
\\
=& \tfrac{1}{2}  \pi_1\big[ T^h[\rx_{s,t}] \big] \otimes \pi_1\big[ T^h[\rx_{s,t}] \big]. 
\end{align*}
\end{proof}

\begin{lemma}
\label{lem:RPTranslation}
The homogeneous rough path metric $d_\alpha$ is $T^h$-invariant. 
\end{lemma}

\begin{proof}
Using that $\#$ is a Group homomorphism, we have 
\begin{align*}
\Big( T^h(\rx)_{s, t} \Big)^{-1} \boxtimes T^h(\ry)_{s, t} =& \Big( \#\Big[ S_{2}(\rx \oplus h)\Big] \Big)_{s, t}^{-1} \boxtimes \# \Big[ S_{2}( \ry \oplus h)\Big]
\\
=&\#\Big[ S_{2}(\rx \oplus h)_{s, t}^{-1} \boxtimes S_{2}(\ry\oplus h)_{s, t}\Big]
\\
=&\#\Big[ S_{2}\Big( \rx^{-1} \boxtimes \ry \oplus 0\Big)_{s, t}\Big]
= \rx^{-1}_{s, t} \boxtimes \ry_{s, t}. 
\end{align*}
Thus
\begin{align*}
d\Big( T^h(\rx)_{s, t}, T^y(\ry)_{s, t}\Big) =& \Big\| \big( T^h(\rx)_{s, t} \big)^{-1} \boxtimes T^h(\ry)_{s, t} \Big\|
\\
=&\Big\| \rx^{-1}_{s, t} \boxtimes \ry_{s, t} \Big\|
= d\Big( \rx_{s, t}, \ry_{s, t}\Big). 
\end{align*}
\end{proof}

\subsubsection{The lift of a Gaussian process}

Gaussian processes have a natural lift for their signature. It is shown in \cite{frizhairer2014} that one can solve the iterated integral of a Gaussian process by approximating the process pathwise and showing that the approximation converges in mean square and almost surely. In particular, the iterated integral of a Gaussian process is an element on the second Wiener-It\^o chaos expansion. 



\begin{assumption}
\label{assumption:GaussianRegularity}
Let $\cL^W$ be the law of a $d$-dimensional, continuous centred Gaussian process with independent components and covariance covariance operator $\cR$. We denote
\begin{align*}
&\cR^{W} \begin{pmatrix}s,& t\\ u,& v\end{pmatrix} = \bE\Big[ W_{s, t} \otimes W_{u, v} \Big] \in \bR^d \otimes \bR^d, 
\quad
\cR^{W}_{s, t} = \bE\Big[ W_{s, t} \otimes W_{s, t} \Big]. 
\\
&\| \cR^W \|_{\varrho-var; [s,t]\times [u, v]} = \Bigg( \sup_{\substack{D \subseteq [s, t]\\ D' \subseteq [s',t']}} \sum_{\substack{i; t_i \in D \\ j; t'_j \in D'}} \Big| \cR\begin{pmatrix}t_i,& t_{i+1} \\ t_{i}',& t_{i+1}'\end{pmatrix} \Big|^{\rho} \Bigg)^{\tfrac{1}{\varrho}}. 
\end{align*}

We assume that $\exists \varrho \in[1, 3/2)$, $\exists M<\infty$ such that
$$
\| \cR^W \|_{\varrho-var; [s, t]^{\times 2}} \leq M \cdot |t-s|^{1/ \varrho}. 
$$
\end{assumption}

Under Assumption \ref{assumption:GaussianRegularity}, it is well known that one can lift a Gaussian process to a Gaussian rough path taking values in $WG\Omega_\alpha(\bR^d)$. 

\newpage
\section{Enhanced Gaussian inequalities}
\label{section:EnGaussInequal}

Let $\bB_\alpha(\rh, \varepsilon):= \{ \rx \in WG\Omega_\alpha (\bR^d):  d_{\alpha}( \rh, \rx)<\varepsilon\}$. In this section, we prove a series of inequalities of Gaussian measures that we will use when proving the small ball probability results of Section \ref{section:SmallBallProbab}. 

\subsection{Translation Inequalities}

\begin{lemma}[Anderson's inequality for Gaussian rough paths]
\label{lem:AndersonInequalityRP}
Let $\cL^W$ be a Gaussian measure and let $\cL^\rw$ be the law of the lift to the Gaussian rough path. Then $\forall \rx\in WG\Omega_{\alpha} (\bR^d)$ 
\begin{equation}
\label{eq:AndersonInequalityRP}
\cL^\rw \Big[ \bB_\alpha(\rx, \varepsilon) \Big] \leq \cL^\rw\Big[ \bB_\alpha( \rId, \varepsilon) \Big]. 
\end{equation}
\end{lemma}

\begin{proof}
See for instance \cite{lifshits2013gaussian}. 
\end{proof}

\begin{lemma}[Cameron-Martin formula for rough paths]
\label{lem:CamMartinFormRP}
Let $\cL^W$ be a Gaussian measure satisfying Assumption \ref{assumption:GaussianRegularity} and let $\cL^\rw$ be the law of the lift to the Gaussian rough path. Let $h\in \cH$ and denote $\rh=S_2[h]$. Then
$$
\cL^\rw\Big[ \bB_\alpha(\rh, \varepsilon) \Big] \geq \exp\left( \frac{-\|h\|_\cH^2}{2} \right) \cL^\rw \Big[  \bB_\alpha(\rId, \varepsilon) \Big]
$$
\end{lemma}

\begin{proof}
Using that the map $W\in C^{\alpha, 0}([0,T]; \bR^d) \mapsto \rw \in WG\Omega_\alpha(\bR^d)$ is measurable, we define the pushforward measure $\cL^\rw = \cL^W \circ \rw^{-1}$.  

We observe that the set $\{ y\in E: d_\alpha(\rw(y), 1)<\varepsilon\}$ is symmetric around $0$. Applying Lemma \ref{lem:RPTranslation} and the Cameron-Martin transform (see \cite{kuelbs1994gaussian}*{Theorem 2}), 
\begin{align*}
\cL^\rw \Big[ \bB_\alpha(\rh, \varepsilon)\Big] =& \cL^W\Big[ \{ x\in E: d_\alpha( \rw(x), \rh )<\varepsilon \} \Big] =\cL^W\Big[ \{ x\in E: d_\alpha( T^{-h} (\rw(x)), T^{-h}(\rh) )<\varepsilon \} \Big]
\\
=& \cL^W\Big[ \{ x\in E: d_\alpha( \rw(x-h), \rId )<\varepsilon \} \Big] = \cL^W\Big[ \{ y\in E: d_\alpha( \rw(y), \rId )<\varepsilon \}+h \Big]
\\
\geq& \exp\Big( \frac{-\|h\|_\cH^2}{2}\Big) \cL^W\Big[ \{ y\in E: d_\alpha( \rw(y), \rId )<\varepsilon \} \Big]
\\
\geq& \exp\Big( \frac{-\|h\|_\cH^2}{2}\Big) \cL^\rw \Big[ \bB_\alpha(\rId, \varepsilon) \Big]. 
\end{align*}
\end{proof}

\begin{definition}
\label{definition:Freidlin-Wentzell_Function}
Let $\cL^W$ be a Gaussian measure and let $\cL^\rw$ be the law of the lift to the Gaussian rough path. We define the Freidlin-Wentzell Large Deviations rate function by
\begin{align*}
I^\rw(\rx):=& \begin{cases}
\frac{\| \pi_1(\rx)\|_\cH^2}{2}, & \mbox{ if } \pi_1(\rx) \in \cH \\
\infty, & \mbox{otherwise.}
\end{cases}
\\
I^\rw(\rx, \varepsilon):=& \inf_{d_{\alpha; [0,T]}(\rx, \ry) <\varepsilon} I(\ry)
\end{align*}
\end{definition}

The following Corollary is similar to a result first proved in \cite{li2001gaussian} for Gaussian measures. 

\begin{corollary}
\label{cor:CamMartinFormRP}
Let $\cL^W$ be a Gaussian measure satisfying Assumption \ref{assumption:GaussianRegularity} and let $\cL^\rw$ be the law of the lift to the Gaussian rough path. Then for $\ry \in \overline{S_2(\cH)}^{d_\alpha}$ and $a\in[0,1]$
$$
\cL^\rw \Big[ \bB_\alpha(\ry, \varepsilon)\Big] \geq \exp\Big( I^\rw(\ry, a\varepsilon)\Big) \cL^\rw\Big[ \bB_\alpha \big(\rId, (1-a)\varepsilon\big)\Big]
$$
\end{corollary}

\begin{proof}
Using that the lift of the RKHS is dense in the support of the Gaussian rough path, we know that there must exist at least on $h\in \cH$ such that $d_\alpha( \rh, \ry)<a\varepsilon$ for any choice of $a\in[0,1]$. Further, by nesting of sets
$$
\cL^\rw\Big[ \bB_\alpha(\ry, \varepsilon) \Big] \geq \cL^\rw\Big[ \bB_\alpha\big(\rh, (1-a)\varepsilon\big) \Big]. 
$$
Now apply Lemma \ref{lem:CamMartinFormRP} and take a minimum over all possible choices of $\rh$. 
\end{proof}

Finally, we recall a useful inequality stated and proved in \cite{cass2013integrability}. 

\begin{lemma}[Borell's rough path inequality]
\label{lem:BorellInequalRP}
Let $\cL^W$ be a Gaussian measure and let $\cL^\rw$ be the law of the lift to the Gaussian rough path. Let $\cK\subset \cH$ be the unit ball with respect to $\| \cdot \|_\cH$ and denote $\rk=\{ \rh=S_2[h]: h\in \cH\}$. 

Let $A$ be a Borel subset of $WG\Omega_\alpha (\bR^d)$, $\lambda>0$ and define
$$
T[A,\delta_\lambda(\rk)]:= \Big\{ T^{h}( \rx): \rx\in A, \rh\in \delta_\lambda(\rk) \Big\}. 
$$
Then, denoting by $\cL_*^\rw$ the inner measure, we have
$$
\cL_*^\rw\Big[ A+ \delta_\lambda (\rk) \Big] \geq \Phi\Big( \lambda  + \Phi^{-1}( \cL^\rw[A])\Big).
$$
\end{lemma}

\subsection{Gaussian correlation inequalities}

Given an abstract Wiener space $(E, \cH, i)$, we consider the element $h\in \cH$ as a random variable on the probability space $(E, \cB(E), \cL)$ where $\cB(E)$ is the cylindrical $\sigma$-algebra generated by the elements of $E^*$. When $E$ is separable, $\cB(E)$ is equal to the Borel $\sigma$-algebra. 

The following Lemma can be found in \cite{stolz1996some}*{Lemma 3.2}. We omit the proof. 
\begin{lemma}
\label{lemma:Stolz_Lem}
Let $(d_k)_{k \in \bN_0}$ be a non-negative sequence such that $d_0 = 1$ and 
$$
\sum_{k=1}^\infty d_k = d <\infty. 
$$
Let $(E, \cH, i)$ be an abstract Wiener space with Gaussian measure $\cL$. Let $(h_i)_{i=1, ..., n} \in \cH$ such that
$$
\| h_i\|_{\cH} = 1, 
\quad 
\Big| \big\langle h_i , h_j \big\rangle_\cH \Big| \leq d_{|i-j|}. 
$$
Then $\exists M_1, M_2>0$ depending on $d$ such that $\forall n\in \bN$, 
$$
\bP\bigg[ \tfrac{1}{n} \sum_{k=1}^n |h_i| \leq \frac{1}{M_1} \bigg] \leq \exp\Big( \tfrac{-n}{M_2} \Big). 
$$
\end{lemma}

This next Lemma, often referred to as \v{S}id\'ak's Lemma, was proved independently in \cites{sidak1968multivariate} and \cite{ khatri1967certain}. 

\begin{lemma}
\label{Lemma:Sidak1}
Let $(E, \cH, i)$ be an abstract Wiener space with Gaussian measure $\cL$ and let $I$ be a countable index. Suppose $\forall j\in I$ that $h_j\in \cH$ and $\varepsilon_j>0$. Then for any $j\in I$
\begin{equation}
\label{eq:Lemma:Sidak1.1}
\cL\bigg[  \bigcap_{i\in I} \Big\{|h_i|<\varepsilon_i\Big\} \bigg] \geq \cL\bigg[ \Big\{ |h_j|<\varepsilon_j \Big\}\bigg] \cL\bigg[ \bigcap_{i\in I\backslash\{j\}} \Big\{ |h_i|<\varepsilon_i\Big\} \bigg]. 
\end{equation}
Equivalently, for $h'_k\in \cH$ such that $\| h_k\|_\cH = \| h'_k\|_\cH$ and $\forall j\in I\backslash \{k\}, \langle h_j, h'_k\rangle_\cH = 0$ then
\begin{equation}
\label{eq:Lemma:Sidak1.2}
\cL\bigg[  \bigcap_{j\in I} \Big\{|h_j|<\varepsilon_j\Big\} \bigg] \geq \cL\bigg[ \bigcap_{j\in I\backslash\{k\}} \Big\{ |h_j|<\varepsilon_j\Big\} \cap \Big\{ |h'_k|<\varepsilon_k\Big\} \bigg]. 
\end{equation}
\end{lemma}

For an eloquent proof, see \cite{bogachev1998gaussian}*{Theorem 4.10.3}. In particular, given a Gaussian process $W$, a countable collection of intervals $(s_j, t_j)_{j\in I}$ and bounds $(\varepsilon_j)_{j\in I}$, we have
\begin{align*}
\bP\bigg[ \bigcap_{j\in I} \Big\{ | W_{s_j, t_j}|<\varepsilon_j \Big\} \bigg] \geq& \bP\bigg[ \Big\{ | W_{s_1, t_1}|<\varepsilon_1 \Big\} \bigg] \cdot \bP\bigg[ \bigcap_{\substack{j\in I\\ j\neq 1}} \Big\{ | W_{s_j, t_j}|<\varepsilon_j \Big\} \bigg] 
\\
\geq& \prod_{j\in I} \bP\bigg[ \Big\{ | W_{s_j, t_j}|<\varepsilon_j \Big\} \bigg]. 
\end{align*}
Thus the probability of a sequence of intervals of a Gaussian process sitting on slices is minimised when the Gaussian random variables are all independent. 

This is an example of the now proved Gaussian corellation conjecture (first proved in \cite{royen2014simple}) which states
\begin{equation}
\label{eq:GaussCorrConj}
\bP\bigg[ \bigcap_{j\in I} \Big\{ | W_{s_j, t_j}|<\varepsilon_j \Big\} \bigg] \geq \bP\bigg[ \bigcap_{j\in I_1} \Big\{ | W_{s_j, t_j}|<\varepsilon_j \Big\} \bigg] \cdot \bP\bigg[ \bigcap_{j\in I_2} \Big\{ | W_{s_j, t_j}|<\varepsilon_j \Big\} \bigg] 
\end{equation}
where $I_1 \cup I_2 = I$ and $I_1 \cap I_2 =\emptyset$. 

Given a pair of abstract Wiener spaces $(E_1, \cH_1, i_1)$ and $(E_2, \cH_2, i_2)$, we can define a Gaussian measure on the Cartesian product $E_1 \oplus E_2$ which has RKHS $\cH_1 \oplus \cH_2$ by taking the product measure $\cL_1 \times \cL_2$ over $(E_1 \oplus E_2, \cB(E_1) \otimes \cB(E_2)) $. 

We define the tensor space $E_1 \otimes_\varepsilon E_2$ of $E_1$ and $E_2$ to be the closure of the algebraic tensor $E_1 \otimes E_2$ with respect to the injective tensor norm
$$
\varepsilon(x):= \sup\Big\{ | (f \otimes g)(x)|: f\in E_1^*, g\in E_2^*, \|f\|_{E_1^*} = \| g\|_{E_2^*} = 1\Big\}. 
$$
Let $f\in (E_1 \otimes_\varepsilon E_2)^*$. Then the map $E_1 \oplus E_2 \ni (x, y) \mapsto f(x\otimes y)$ is measurable and the pushforward of $f$ with respect to the Gaussian measure is an element of the second Wiener It\^o chaos.  In the case where the tensor product is of two Hilbert spaces, there is no question over the choice of the norm for $\cH_1 \otimes \cH_2$. 

A problem similar to this was first studied in \cite{ledoux2002large}. We emphasise that our result is a lot more general. 
\begin{lemma}
\label{Lemma:Sidak2}
Let $(E_1, \cH_1, i_1)$ and $(E_2, \cH_2, i_2)$ be abstract Wiener spaces with Gaussian measures $\cL_1$ and $\cL_2$. Let $\cL_1 \times \cL_2$ be the product measure over the direct sum $E_1 \oplus E_2$. Let $I_1, I_2,  I_3$ be countable indexes. Suppose that $\forall j\in I_1, h_{j,1} \in \cH_1$ and $\varepsilon_{j,1}>0$, $\forall j\in I_2, h_{j,2} \in \cH_2$ and $\varepsilon_{j, 2}>0$, and $\forall j\in I_3, h_{j,3} \in \cH_1 \otimes \cH_2$ and $\varepsilon_{j, 3}>0$. Additionally, denote $\hat{\otimes} :E_1 \oplus E_2 \to E_1 \otimes_\varepsilon E_2$ by $\hat{\otimes}(x, y) = x\otimes y$. Then
\begin{align*}
(\cL_1 \times \cL_2)&\bigg[ \bigcap_{j\in I_1} \Big\{ |h_{j,1}|<\varepsilon_{j, 1}\Big\} \bigcap_{j\in I_2} \Big\{ |h_{j, 2}|< \varepsilon_{j, 2}\Big\} \bigcap_{j\in I_3} \Big\{ |h_{j, 3}(\hat{\otimes})|<\varepsilon_{j,3}\Big\} \bigg] 
\\
\geq &\prod_{j \in I_1} \cL_1\bigg[ \Big\{ |h_{j, 1}|<\varepsilon_{j, 1}\Big\}\bigg] \cdot \prod_{j\in I_2} \cL_2\bigg[\Big\{ |h_{j, 2}|<\varepsilon_{j, 2}\Big\}\bigg] \cdot \prod_{j \in I_3} (\cL_1\times \cL_2)\bigg[ \Big\{ |h_{j,3}(\hat{\otimes})|<\varepsilon_{j, 3}\Big\}\bigg] 
\end{align*}
\end{lemma}

\begin{proof}
When $I_3=\emptyset$, Lemma \ref{Lemma:Sidak2} comes immediately by applying Lemma \ref{Lemma:Sidak1}. When $I_3\neq \emptyset$, the bilinear forms $f_{j, 3}(\hat{\otimes})$ are not bounded on $E_1 \oplus E_2$ and so we cannot immediately apply Lemma \ref{Lemma:Sidak1} (they are bounded on the space $E_1 \otimes_\varepsilon E_2$). 

However, we do have that for $y\in E_2$ fixed, the functional $x\mapsto h( x\otimes y)$ is a linear functional and for $x\in E_1$ fixed, the functional $y\mapsto h( x\otimes y)$ is a linear functional (not necessarily bounded functionals). Thus, by the definition of the product measure, we have
\begin{align*}
(\cL_1 \times \cL_2)&\bigg[ \bigcap_{j\in I_1} \Big\{ |h_{j,1}|<\varepsilon_{j, 1}\Big\} \bigcap_{j\in I_2} \Big\{ |h_{j, 2}|< \varepsilon_{j, 2}\Big\} \bigcap_{j\in I_3} \Big\{ |h_{j, 3}(\hat{\otimes})|<\varepsilon_{j,3}\Big\} \bigg] 
\\
=&\int_{E_2}\int_{E_1} \prod_{j\in I_1} \1_{\{|h_{j,1}|<\varepsilon_{j, 1} \}}(x) \cdot \prod_{j\in I_2} \1_{\{|h_{j,2}|<\varepsilon_{j, 2}\}}(y) \cdot \prod_{j\in I_3} \1_{\{|h_{j,3}(\hat{\otimes })|<\varepsilon_{j,3}\}}(x,y) d\cL_1(x) d\cL_2(y)
\\
\geq&\int_{E_2}\prod_{j\in I_2} \1_{\{|h_{j,2}|<\varepsilon_{j, 2}\}}(y) \cdot \int_{E_1} \prod_{j\in I_1} \1_{\{|h'_{j,1}|<\varepsilon_{j, 1} \}}(x) \cdot \prod_{j\in I_3} \1_{\{|h'_{j,3}(\hat{\otimes })|<\varepsilon_{j,3}\}}(x,y) d\cL_1(x) d\cL_2(y)
\end{align*}
where for each $j\in I_1$ $\|h'_{j, 1}\|_\cH = \| h_{j, 1}\|_\cH$, for each $j\in I_3$ and $y\in E_2$ fixed $\|h'_{j,3}(\cdot \otimes y)\|_\cH = \|h_{j,3}(\cdot \otimes y)\|_\cH$, and the vectors $\{h'_{j, 1}\}_{j\in I_1}\cup \{h'_{j, 3}(\cdot \otimes y)\}_{j\in I_3}$ are orthonormal in $\cH$. This comes from applying Equation \eqref{eq:Lemma:Sidak1.2} from Lemma \ref{Lemma:Sidak1}. 

Similarly, swapping the order of integration and repeating yields
\begin{align*}
\geq&\int_{E_1}\prod_{j\in I_1} \1_{\{|h'_{j,1}|<\varepsilon_{j, 1} \}}(x) \cdot  \int_{E_2} \prod_{j\in I_2} \1_{\{|h'_{j,2}|<\varepsilon_{j, 2}\}}(y) \cdot \prod_{j\in I_3} \1_{\{|h''_{j,3}(\hat{\otimes })|<\varepsilon_{j,3}\}}(x,y) d\cL_2(y) d\cL_1(x) 
\\
\geq& \prod_{j \in I_1} \cL_1\bigg[ \Big\{ |h'_{j, 1}|<\varepsilon_{j, 1}\Big\}\bigg] \cdot \prod_{j\in I_2} \cL_2\bigg[\Big\{ |h'_{j, 2}|<\varepsilon_{j, 2}\Big\}\bigg] \cdot \prod_{j \in I_3} (\cL_1\times \cL_2)\bigg[ \Big\{ |h''_{j,3}(\hat{\otimes})|<\varepsilon_{j, 3}\Big\}\bigg] 
\end{align*}
where for each $j\in I_2$ $\|h'_{j,2}\|_\cH = \|h_{j,2}\|_\cH$, for each $j\in I_3$ and $x\in E_1$ fixed $\| h'_{j, 3}(x \otimes \cdot)\|_\cH = \| h''_{j,3}(x \otimes \cdot)\|_\cH$, and the vectors $\{ h'_{j, 2}\}_{j\in I_2} \cup \{h''_{j,3}(x \otimes \cdot)\}_{j\in I_3}$ are orthonormal in $\cH$. 
\end{proof}

In fact, rather than dividing this intersection of sets into a product of probabilities completely (as will be necessary later in this paper), we could have used Equation \eqref{eq:GaussCorrConj} to divide the intersection into the product of any number of two intersections. We do not state this to avoid writing already challenging notation and because there is no need for such a result in Section \ref{section:SmallBallProbab}. 

\begin{proposition}[\v{S}id\'ak's Lemma for higher order Wiener-It\^o chaos elements.]
\label{pro:Sidak}
Let $m$ be a positive integer. Let $(E_1, \cH_1, i_1)$, ..., $(E_m, \cH_m, i_m)$ be $m$ abstract Wiener spaces with Gaussian measures $\cL_1$, ..., $\cL_m$. Let $\cL_1 \times ... \times \cL_m$ be the product measure over the direct sum $E_1 \oplus ... \oplus E_m$. Let $I_1$, $I_2$, ..., $I_m$ be $m$ countable indexes. Suppose that for $l\in\{ 1, ..., m\}$, $\forall j\in I_l$
$$
h_{j, l} \in \bigcup_{\substack{k_1, ..., k_l\\ k_1 \neq ... \neq k_l}} \cH_{k_1} \otimes ... \otimes \cH_{k_l}, \quad \varepsilon_{j, l}>0. 
$$
Next, suppose
$$
\hat{\otimes}_l:E_{k_1} \oplus ... \oplus E_{k_l} \to E_{k_1} \otimes_\varepsilon ... \otimes_\varepsilon E_{k_l}, \quad \hat{\otimes}(x_{k_1}, ..., x_{k_l}) := x_{k_1} \otimes ... \otimes x_{k_l}
$$
Then
\begin{align*}
\Big(\cL_1 \times ... \times \cL_m\Big)\Bigg[ \bigcap_{l=1}^m \bigcap_{j\in I_l} \Big\{ |h_{j, l}(\hat{\otimes}_l)|<\varepsilon_{j, l}\Big\} \Bigg] \geq \prod_{l=1}^m\prod_{j\in I_l} \Big( \cL_1 \times ... \times \cL_m\Big) \Bigg[ \Big\{ |h_{j, l}(\hat{\otimes}_l)|<\varepsilon_{j, l}\Big\} \Bigg]
\end{align*}
\end{proposition}

\begin{proof}
Extensive applications of the methods of Lemma \ref{Lemma:Sidak2} and Equation \eqref{eq:Lemma:Sidak1.2}. 
\end{proof}

\section{Small ball probabilities for enhanced Gaussian processes}
\label{section:SmallBallProbab}

The first result, the main endeavour of this paper, demonstrates that for Gaussian rough paths the SBPs cannot converge ``faster'' than a function of the regularity of the covariance. 

\begin{theorem}
\label{Thm:SmallBallProbab}
Let $\cL^W$ be a Gaussian measure satisfying Assumption \ref{assumption:GaussianRegularity} for some $\varrho\in[1, 3/2)$ and let $\rw$ be the lifted Gaussian rough path. Then for $\tfrac{1}{3} <\alpha<\tfrac{1}{2\varrho}$ we have
\begin{equation}
\label{eq:ThmSmallBallProbab}
\fB(\varepsilon):=-\log\Big( \bP\Big[ \| \rw \|_\alpha <\varepsilon\Big] \Big) \lesssim \varepsilon^{\frac{-1}{\tfrac{1}{2\varrho} - \alpha}}. 
\end{equation}
\end{theorem}

Secondly, we demonstrate that provided the covariance of a Gaussian process is adequately irregular, the small ball probabilities cannot converge slower than this same speed. 

\begin{proposition}
\label{pro:SmallBallProbab-Lower}
Let $\cL^W$ be a Gaussian measure and for $s, t\in [0,T]$, let $\bE\Big[ \big| W_{s, t} \big|^2 \Big] = \sigma^2\Big( \big| t-s \big| \Big)$. 

Suppose that $\exists h>0$ such that $h\leq T$ and
\begin{enumerate}
\item $\exists C_1>0$ such that $\forall \tau \in [0,h)$, 
\begin{equation}
\label{eq:Thm:SmallBallProbab-Lower-1}
\sigma^2\Big( \tau \Big) \geq C_1 \cdot |\tau|^{\tfrac{1}{\varrho}}. 
\end{equation}
\item Suppose that there exits $0<C_2<4$ such that for $\tau \in [0,\tfrac{h}{2})$, 
\begin{equation}
\label{eq:Thm:SmallBallProbab-Lower-2}
\sigma^2\Big( 2\tau\Big) \leq C_2 \cdot \sigma^2\Big( \tau\Big). 
\end{equation}
\item Suppose that $\sigma^2$ is three times differentiable and there exists a constant $C_3>0$ such that $\forall \tau \in [0,h)$, 
\begin{equation}
\label{eq:Thm:SmallBallProbab-Lower-3}
\Big| \nabla^3 \big[ \sigma^2 \big] \Big( \tau \Big) \Big| \leq \frac{C_3}{\tau^{3 - \tfrac{1}{\varrho}}}. 
\end{equation}
\end{enumerate}
Then
$$
-\log\Big( \bP\Big[ \| W \|_\alpha < \varepsilon \Big] \Big) \gtrsim \varepsilon^{\tfrac{-1}{\tfrac{1}{2\varrho} - \alpha}} . 
$$
\end{proposition}

\begin{theorem}
\label{Thm:SmallBallProbab-Both}
Let $\tfrac{1}{3} <\alpha<\tfrac{1}{2\varrho}$. Let that $\cL^W$ be a Gaussian measure and $\forall s, t\in [0,T]$, let $\bE\Big[ \big| W_{s, t} \big|^2 \Big] = \sigma^2\Big( \big| t-s \big| \Big)$. 

Suppose that $\exists h>0$ and $c_1, c_2>0$ such that $\forall \tau \in [0,h)$, $\sigma^2(\tau)$ is convex and 
\begin{equation}
\label{eq:Thm:SmallBallProbab-Both}
c_1 \cdot |\tau|^{\tfrac{1}{\varrho}} \leq \sigma^2 \Big( \tau \Big) \leq c_2 \cdot | \tau |^{\tfrac{1}{\varrho}}. 
\end{equation}
Then $\cR^W$ satisfies Assumption \ref{assumption:GaussianRegularity}. Suppose additionally that $\frac{c_2 \cdot 2^{\tfrac{1}{\varrho}}}{c_1}<4$. Then Equation \eqref{eq:Thm:SmallBallProbab-Lower-2} is satisfied. 

Additionally, if $\sigma^2$ satisfies Equation \eqref{eq:Thm:SmallBallProbab-Lower-3}, then
$$
\fB(\varepsilon) \approx \varepsilon^{\frac{-1}{\tfrac{1}{2\varrho} - \alpha}}. 
$$
\end{theorem}

\begin{example}
Fractional Brownian motion is the Gaussian process with covariance
$$
\bE\Big[ W^{H}_t \otimes W^{H}_s\Big] = \tfrac{d \cdot I_d}{2}\Big( |t|^{2H} + |s|^{2H} - |t-s|^{2H}\Big), 
$$
where $I_d$ is the $d$-dimensional identity matrix. It is well known that the covariance of Fractional Browian motion satisfies 
$$
\bE\Big[ \big| W^{H}_{s, t} \big|^2 \Big] = d \cdot \big| t-s \big|^{2H}, 
$$
so that the assumptions of Theorem \ref{Thm:SmallBallProbab-Both} are satisfied. The small ball probabilities of Fractional Brownian motion were studied in \cite{kuelbs1995small} with respect to the H\"older norm using that it has stationary increments, and Theorem \ref{Thm:SmallBallProbab-Both} extends these results to enhanced fractional Brownian motion when $H \in (\tfrac{1}{3}, \tfrac{1}{2})$. 

Further, the upper and lower bounds of Equation \eqref{eq:Thm:SmallBallProbab-Both} are only required locally around 0, so these results also apply for the fractional Brownian bridge (which fails to satisfy \eqref{eq:Thm:SmallBallProbab-Both} for $\tau>T/2$). 
\end{example}

\subsection{Preliminaries}

Firstly, we address a method for discretising the rough path H\"older norm. To the best of the authors knowledge, this result has not previously been stated in the framework of rough paths. The proof is an adaption of the tools used in \cite{kuelbs1995small}*{Theorem 2.2}. 

\begin{lemma}[Discretisation of Rough path norms]
\label{lemma:DiscretisationNorm}
Let $\rw \in WG\Omega_\alpha(\bR^d)$ be a rough path. Then we have
\begin{align*}
\| \rw \|_\alpha &\leq \max\Bigg( 2\sum_{l=1}^\infty \sup_{ i=1, ..., 2^l} \frac{\| \rw_{(i-1)T2^{-l}, iT2^{-l}} \|_{cc}}{\varepsilon^\alpha} , 
\\
& \qquad \qquad 3 \sup_{j\in \bN_0 } \sup_{ i=0, ... \left\lfloor \tfrac{2^j(T-\varepsilon)}{\varepsilon}\right\rfloor } \sum_{l=j+1}^\infty \sup_{m=1, ..., 2^{l-j}} \frac{ \| \rw_{(m-1)2^{-l} + i \varepsilon 2^{-j}, m2^{-l} + i \varepsilon 2^{-j}} \|_{cc} }{ \varepsilon^\alpha 2^{-\alpha(j+1)} } \Bigg).
\end{align*}
\end{lemma}

\begin{proof}
Let $0< \varepsilon <T$. 
\begin{equation}
\label{eq:lemma:DiscretisationNorm1}
\| \rw\|_\alpha \leq \sup_{\substack{ s, t\in [0,T]\\ |t-s|\geq\varepsilon}} \frac{ \| \rw_{s, t}\|_{cc}}{|t-s|^{\alpha}} \bigvee \sup_{\substack{ s, t\in[0,T] \\ |t-s|<\varepsilon}} \frac{ \| \rw_{s, t}\|_{cc}}{|t-s|^{\alpha}} 
\leq \sup_{s\in[0,T]} \frac{2 \| \rw_{0,s}\|_{cc}}{\varepsilon^\alpha} \bigvee \sup_{\substack{0\leq s \leq T\\ 0\leq t \leq \varepsilon \\ |s+t|<T}} \frac{\| \rw_{s, s+t}\|_{cc}}{ |t|^\alpha}
\end{equation}

Firstly, writing $s\in[0,T]$ as a sum of dyadics and exploiting the sub-additivity of the Carnot-Carath\'eodory norm, we get
\begin{equation}
\label{eq:lemma:DiscretisationNorm1.1}
\| \rw_{0,s}\|_{cc} \leq \sum_{l=1}^\infty \sup_{i=1, ..., 2^l} \| \rw_{(i-1)T2^{-l}, iT2^{-l}} \|_{cc}. 
\end{equation}
Hence
\begin{equation}
\label{eq:lemma:DiscretisationNorm1.2}
\sup_{s\in[0,T]} \frac{2 \| \rw_{0,s}\|_{cc}}{\varepsilon^\alpha} \leq \sum_{l=1}^\infty \sup_{i=1, ..., 2^l} \frac{ 2\| \rw_{(i-1)T2^{-l}, iT2^{-l}} \|_{cc}}{\varepsilon^\alpha} . 
\end{equation}

Secondly,
\begin{align*}
\sup_{\substack{0\leq s \leq T\\ 0\leq t \leq \varepsilon}} \frac{\| \rw_{s, s+t}\|_{cc}}{ |t|^\alpha} \leq& \sup_{s\in [0,T]} \max_{j\in \bN_0} \sup_{\varepsilon 2^{-j-1}\leq t<\varepsilon 2^{-j}}  \frac{\| \rw_{s, s+t}\|_{cc}}{ |t|^\alpha}
\\
\leq& \max_{j\in \bN_0}  \sup_{s\in [0,T]}  \sup_{0<t<\varepsilon 2^{-j}} \frac{\| \rw_{s, s+t}\|_{cc}}{ |\varepsilon|^\alpha \cdot 2^{-\alpha(j+1)}}
\\
\leq& \max_{j\in \bN_0}  \max_{i=0, ..., \left\lfloor \tfrac{2^j(T-\varepsilon)}{\varepsilon}\right\rfloor }  \sup_{0<t<\varepsilon 2^{-j}} \frac{3 \| \rw_{i\varepsilon 2^{-j}, i\varepsilon 2^{-j}+t}\|_{cc}}{ |\varepsilon|^\alpha \cdot 2^{-\alpha(j+1)}}. 
\end{align*}

Then, as with Equation \eqref{eq:lemma:DiscretisationNorm1.1} we have that for $t\in (0, \varepsilon 2^{-j})$, 
$$
\| \rw_{i\varepsilon 2^{-j}, i\varepsilon 2^{-j}+t}\|_{cc} \leq \sum_{l=j+1}^\infty \sup_{m=1, ..., 2^{l-j}} \| \rw_{(m-1)2^{-l} + i \varepsilon 2^{-j}, m2^{-l} + i \varepsilon 2^{-j}} \|_{cc}. 
$$
Hence
\begin{equation}
\label{eq:lemma:DiscretisationNorm1.3}
\sup_{\substack{0\leq s \leq T\\ 0\leq t \leq \varepsilon \\ |s+t|<T}} \frac{\| \rw_{s, s+t}\|_{cc}}{ |t|^\alpha} \leq \sup_{j\in \bN_0 } \sup_{ i=0, ... \left\lfloor \tfrac{2^j(T-\varepsilon)}{\varepsilon}\right\rfloor } \sum_{l=j+1}^\infty \sup_{m=1, ..., 2^{l-j}} \frac{ 3 \| \rw_{(m-1)2^{-l} + i \varepsilon 2^{-j}, m2^{-l} + i \varepsilon 2^{-j}} \|_{cc} }{ \varepsilon^\alpha 2^{-\alpha(j+1)} }. 
\end{equation}
Combining Equation \eqref{eq:lemma:DiscretisationNorm1} with Equation \eqref{eq:lemma:DiscretisationNorm1.2} and Equation \eqref{eq:lemma:DiscretisationNorm1.3} yields the result. 
\end{proof}

\subsection{Proof of Theorem \ref{Thm:SmallBallProbab}}

The proof of the upper bound is the main contribution of this Section. 

\begin{proof}

Let $n_0$ be a positive integer such that $\varepsilon^{-1} \leq 2^{n_0} \leq 2\varepsilon^{-1}$ and denote $\beta = \tfrac{1}{2\varrho} - \alpha$ for brevity. Define
\begin{align}
\nonumber
\varepsilon_l^{(1)}:=& \Big( \frac{3}{2} \Big)^{\tfrac{-|l-n_0|}{2\varrho}} \varepsilon^{\tfrac{1}{2\varrho}} \cdot \frac{(1-2^{-\beta/2})}{4}, 
\\
\label{eq:ThmSmallBallProbab3.2}
\varepsilon_{j,l}^{(2)}:=& \frac{\varepsilon^{\beta} 2^{\tfrac{-l}{2\varrho}} }{3} \cdot \frac{2^{\beta(l+j)/2} (1-2^{-\beta/2})}{2^{-\alpha(j+1)} }.
\end{align}
Observe that these satisfy the properties 
\begin{align*}
\sum_{l=1}^\infty \varepsilon_l^{(1)} \leq \frac{ \varepsilon^{\tfrac{1}{2\varrho}} }{2} \quad\textrm{and}\quad
\sum_{l=j+1}^\infty \varepsilon_{j,l}^{(2)} \leq \frac{ \varepsilon^\beta }{3}. 
\end{align*}
Therefore, using Lemma \ref{lemma:DiscretisationNorm} gives the lower bound
\begin{align}
\nonumber
&\bP\Big[ \| \rw\|_\alpha \leq \varepsilon^\beta \Big] 
\geq \bP\Bigg[ \sup_{i=1, ..., 2^l} \| \rw_{(i-1)T2^{-l}, iT2^{-l}} \|_{cc} \leq \varepsilon_l^{(1)} \quad \forall l\in \bN, 
\\
\label{eq:Thm:SmallBallProbab1}
& \sup_{i=0, ..., \left\lfloor \tfrac{2^{j}(T-\varepsilon)}{\varepsilon}\right\rfloor} \sup_{m=1, ..., 2^{l-j}} \frac{\| \rw_{(m-1)2^{-l}\varepsilon + i2^{-j}\varepsilon, m2^{-l}\varepsilon + i2^{-l} \varepsilon} \|_{cc}}{ \varepsilon^\alpha 2^{-\alpha(j+1)}} \leq \varepsilon^{(2)}_{j, l} \quad \forall l\geq j+1, j,l\in \bN_0 \Bigg].
\end{align}

Next, using the equivalence of the Homogeneous norm from Equation \ref{eq:HomoNorm}, we have that there exists a constant dependent only on $d$ such that
\begin{align*}
\| \rw_{s, t}\|_{cc} \leq c(d) \cdot \sup_{A\in \cA_2} \Big| \langle \log_{\boxtimes} (\rw_{s, t}), e_A\rangle \Big|^{1/|A|}
\end{align*}
Using that all homogeneous norms are equivalent and Equation \eqref{eq:HomoNorm}, we get the representation
\begin{align*}
\| \rw_{s, t}\|_{cc} \leq& c(d) \sup_{p=1, ..., d} \Big| \Big\langle \rw_{s, t}, e_{p} \Big\rangle \Big| 
\bigvee \sup_{\substack{p,q=1, ..., d\\ p\neq q}} \Big| \Big\langle \rw_{s, t}, e_{p, q} \Big\rangle \Big|^{1/2}
\end{align*}
Applying Proposition \ref{pro:Sidak} to this yields
\begin{align}
\nonumber
\bP\Big[& \| \rw\|_\alpha \leq \varepsilon^\beta\Big] 
\\
\nonumber
\geq& \Bigg\{ \prod_{l=1}^\infty \prod_{i=1}^{2^l} 
\Bigg( 
\prod_{p=1}^d \bP\Big[ \Big| \Big\langle \rw_{(i-1)T2^{-l}, iT2^{-l}}, e_{p}\Big\rangle \Big| \leq \tfrac{\varepsilon_l^{(1)}}{c(d)} \Big] 
\\
\nonumber
&\quad \cdot \prod_{\substack{p,q=1\\ p\neq q}}^d \bP\Big[ \Big| \Big\langle \rw_{(i-1)T2^{-l}, iT2^{-l} }, e_{p,q}\Big\rangle \Big| \leq \Big(\tfrac{\varepsilon_l^{(1)}}{c(d)}\Big)^2 \Big] \Bigg) \Bigg\}
\\
\nonumber
&\times\Bigg\{
\prod_{j=0}^\infty \prod_{l=j+1}^\infty \prod_{i=0}^{\left\lfloor 2^j(T-\varepsilon)/\varepsilon \right\rfloor} \prod_{m=1}^{2^{l-j}} \Bigg( 
\prod_{p=1}^d \bP\Big[ \Big| \Big\langle \rw_{\varepsilon (m-1)2^{-l} + \varepsilon i2^{-j}, \varepsilon m2^{-l} + \varepsilon i2^{-j}}, e_{p}\Big\rangle \Big| \leq \tfrac{\varepsilon_{j,l}^{(2)}\cdot \varepsilon^\alpha 2^{-\alpha(j+1)}}{c(d)} \Big] 
\\
\label{eq:ThmSmallBallProbab1.1}
&\quad \cdot \prod_{\substack{p,q=1\\p\neq q}}^d \bP\Big[ \Big| \Big\langle \rw_{\varepsilon (m-1)2^{-l} + \varepsilon i2^{-j}, \varepsilon m2^{-l} + \varepsilon i2^{-j}}, e_{p,q}\Big\rangle \Big| \leq \Big( \tfrac{\varepsilon_{j,l}^{(2)}\cdot \varepsilon^\alpha 2^{-\alpha(j+1)}}{c(d)}\Big)^2 \Big] \Bigg) \Bigg\}. 
\end{align}

For the terms associated to words of length 1, the computation of this probability under Assumption \ref{assumption:GaussianRegularity} is simply
\begin{align}
\bP\Big[ |W_{s, t}| \leq \varepsilon\Big] =
 \erf\Big( \tfrac{\varepsilon}{\sqrt{2} \bE[ |W_{s, t}|^2]^{1/2}} \Big)
\label{eq:ThmSmallBallProbab4.1}
\geq& \erf\Big( \tfrac{\varepsilon}{\sqrt{2} M |t-s|^{1/2\varrho}} \Big). 
\end{align}
For longer words, we only attain the lower bound. 
\begin{align}
\nonumber
\bP\bigg[ \Big| \Big\langle \int_s^t W_{s, r}\otimes dW_r, e_{p, q}\Big\rangle \Big| < \varepsilon \bigg] =& \bE \Bigg[ \bP\Bigg[ \Big| \Big\langle \int_s^t W_{s, r}\otimes dW_r, e_{p, q}\Big\rangle \Big| < \varepsilon\Bigg| \sigma\Big(\langle W, e_p\rangle \Big) \Bigg] \Bigg]
\\
\nonumber
=& \bE\Big[ \erf\Big( \tfrac{\varepsilon}{\sqrt{2} \| W_{s, \cdot} \1_{(s, t)} \|_\cH} \Big) \Big]
\\
\geq& \erf\Big( \tfrac{\varepsilon}{\sqrt{2} \bE[ \| W_{s, \cdot} \1_{(s, t)} \|_\cH^2]^{1/2} } \Big) 
\label{eq:ThmSmallBallProbab4.2}
\geq \erf\Big( \tfrac{\varepsilon}{\sqrt{2} M|t-s|^{2/(2\varrho)} } \Big). 
\end{align}
We also use the lower bounds
\begin{align}
\label{eq:ThmSmallBallProbab2.1}
\erf\Big( \tfrac{t}{\sqrt{2}} \Big) \geq& \tfrac{t}{2} & \mbox{ for } t\in[0,1],
\\
\label{eq:ThmSmallBallProbab2.2}
\erf\Big( \tfrac{st}{\sqrt{2}} \Big) \geq& \exp\Bigg( \frac{-\exp\Big( \tfrac{-(st)^2}{2}\Big)}{1-\exp\Big( \tfrac{-s^2}{2} \Big)} \Bigg) & \mbox{ for } s>0, t\in [1, \infty).
\end{align}

We now consider the terms from Equation \ref{eq:ThmSmallBallProbab1.1} with the product over $(j, l, i, m)$. By Assumption \ref{assumption:GaussianRegularity}, the expression \eqref{eq:ThmSmallBallProbab3.2} and Equation \ref{eq:ThmSmallBallProbab4.1} we have
\begin{align*}
\bP\Big[ \Big| \Big\langle \rw_{\varepsilon (m-1)2^{-l} + \varepsilon i2^{-j}, \varepsilon m2^{-l} + \varepsilon i2^{-j}}, e_{p}\Big\rangle \Big| \leq \tfrac{\varepsilon_{j,l}^{(2)}\cdot \varepsilon^\alpha 2^{-\alpha(j+1)}}{c(d)} \Big] 
\geq \erf\Bigg( \tfrac{(1-2^{-\beta/2}) }{3Mc(d)} \cdot 2^{\beta(l+j)/2} \Bigg).
\end{align*}
By similarly applying Equation \ref{eq:ThmSmallBallProbab4.2}
\begin{align*}
\bP\Big[& \Big| \Big\langle \rw_{\varepsilon (m-1)2^{-l} + \varepsilon i2^{-j}, \varepsilon m2^{-l} + \varepsilon i2^{-j}}, e_{p,q}\Big\rangle \Big| \leq \Big( \tfrac{\varepsilon_{j,l}^{(2)}\cdot \varepsilon^\alpha 2^{-\alpha(j+1)}}{c(d)}\Big)^2 \Big] 
\geq \erf\Bigg( \Big(\tfrac{(1-2^{-\beta/2}) }{3Mc(d)}\Big)^2 \cdot 2^{2\beta(l+j)/2} \Bigg).
\end{align*}
Next, we denote $s = \frac{(1-2^{-\beta/2)}) }{3Mc(d)}$, apply the lower bound \eqref{eq:ThmSmallBallProbab2.2} and multiply all the terms together correctly to obtain 
\begin{align}
\nonumber
\prod_{j=0}^\infty \prod_{l=j+1}^\infty& \exp\Bigg( \frac{-T2^l}{\varepsilon} \cdot \Bigg[ \frac{d}{1-e^{-s^2/2}} \exp\Big( \tfrac{-s^2}{2} 2^{\beta(l+j)} \Big) + \frac{d(d-1)}{2(1-e^{-s^4/2})} \exp\Big( \tfrac{-s^4}{2} 2^{2\beta(l+j)} \Big) \Bigg] \Bigg)
\\
\label{eq:ThmSmallBallProbab5.1}
&\geq \exp\Bigg( -\frac{c_1(d, T, M, \beta)}{\varepsilon} \Bigg). 
\end{align}

Secondly, we consider the terms from Equation \ref{eq:ThmSmallBallProbab1.1} with the product over $(l,i)$ and restrict ourselves to the case where $l>n_0$. By applying the definition of $n_0$, $\varepsilon_l^{(1)}$ and using Assumption \ref{assumption:GaussianRegularity}
\begin{align*}
\bP\Big[& \Big| \Big\langle \rw_{(i-1)2^{-l}, i2^{-l}}, e_{p}\Big\rangle \Big| \leq \tfrac{\varepsilon_l^{(1)}}{c(d)} \Big] 
\geq \erf\Bigg( \Big(\tfrac{4}{3}\Big)^{\tfrac{l-n_0}{2\varrho}} \cdot \tfrac{(1-2^{-\beta/2})}{4Mc(d) \sqrt{2}} \Bigg).
\end{align*}
Similarly, by using Equations \eqref{eq:ThmSmallBallProbab4.2},
\begin{align*}
\bP\Big[ \Big| \Big\langle \rw_{(i-1)2^{-l}, i2^{-l} }, e_{p,q}\Big\rangle \Big| \leq \Big(\tfrac{\varepsilon_l^{(1)}}{c(d)}\Big)^2 \Big] 
\geq \erf\Bigg( \Big(\tfrac{4}{3}\Big)^{\tfrac{2(l-n_0)}{2\varrho}} \cdot \tfrac{1}{\sqrt{2}} \cdot \Big(\tfrac{(1-2^{-\beta/2})}{4Mc(d)}\Big)^2 \Bigg). 
\end{align*}
Now applying Equation \eqref{eq:ThmSmallBallProbab2.2} and multiplying all the terms together gives
\begin{align}
\nonumber
\prod_{l=n_0+1}^\infty \exp\Bigg(& -2^l\Bigg[ \frac{d}{1-e^{-s^2/2}} \exp\Big( -\tfrac{s^2}{2}\Big(\tfrac{4}{3}\Big)^{\tfrac{l-n_0}{2\varrho}}\Big) 
+ \frac{d(d-1)}{2(1-e^{-s^4/2})} \exp\Big(-\tfrac{s^4}{2} \Big( \tfrac{4}{3}\Big)^{\tfrac{2(l-n_0)}{2\varrho}} \Big) \Bigg] \Bigg)
\\
\label{eq:ThmSmallBallProbab5.2}
\geq& \exp\Big( -2^{n_0} c_2(d, T, M, \beta)\Big) \geq \exp\Big( -\tfrac{2c_2(d, T, M, \beta)}{\varepsilon} \Big)
\end{align}
where $s=\Big(\tfrac{(1-2^{-\beta/2})}{4Mc(d)}\Big)$. 

Finally, we come to the terms from Equation \ref{eq:ThmSmallBallProbab1.1} with the product over $(l,i)$ where we consider the remaining terms for $l=0, ..., n_0$. Using the definition of $\varepsilon$ and Assumption \ref{assumption:GaussianRegularity}
\begin{align*}
\bP\Big[& \Big| \Big\langle \rw_{(i-1)2^{-l}, i2^{-l}}, e_{p}\Big\rangle \Big| \leq \tfrac{\varepsilon_l^{(1)}}{c(d)} \Big] 
\geq 
\erf\Bigg( \Big( \tfrac{1}{3}\Big)^{\tfrac{n_0-l}{2\varrho}} \cdot \tfrac{1}{\sqrt{2}} \cdot\tfrac{1-2^{-\beta/2}}{4Mc(d)} \Bigg).
\end{align*}
Similarly, by using Equations \eqref{eq:ThmSmallBallProbab4.2}, 
\begin{align*}
\bP\Big[ \Big| \Big\langle \rw_{(i-1)2^{-l}, i2^{-l} }, e_{p,q}\Big\rangle \Big| \leq \Big(\tfrac{\varepsilon_l^{(1)}}{c(d)}\Big)^2 \Big] 
\geq 
\erf\Bigg( \Big( \tfrac{1}{3}\Big)^{\tfrac{2(n_0-l)}{2\varrho}} \cdot \tfrac{1}{\sqrt{2}} \cdot \Big( \tfrac{1-2^{-\beta/2}}{4Mc(d)}\Big)^2 \Bigg). 
\end{align*}
For these terms, we use the lower bound \eqref{eq:ThmSmallBallProbab2.1} and multiply all the terms together to get
\begin{align}
\nonumber
\prod_{l=0}^{n_0}& \Bigg( \Bigg[ \tfrac{1}{2} \cdot \tfrac{(1-2^{-\beta/2})}{4Mc(d)} \cdot \Big( \tfrac{1}{3}\Big)^{\tfrac{(n_0-l)}{2\varrho}} \Bigg]^{dT2^l} \cdot 
\Bigg[ \tfrac{1}{2} \cdot \Big(\tfrac{(1-2^{-\beta/2})}{4Mc(d)}\Big)^2 \cdot \Big( \tfrac{1}{3}\Big)^{\tfrac{2(n_0-l)}{2\varrho}} \Bigg]^{\tfrac{d(d-1)T2^l}{2}} \Bigg)
\\
\nonumber
\geq& \exp\Bigg( -2^{n_0} \sum_{l=1}^{n_0} \Bigg[ dT\bigg( 2^{-(n_0-l)} \log\Big( 2\cdot \Big(\tfrac{4Md(c)}{1-2^{-\beta/2}}\Big)\Big) + (n_0-l)\log\Big( 3^{\tfrac{1}{2\varrho}} \Big) \bigg)
\\
\nonumber
&\qquad +\tfrac{d(d-1)}{2} \bigg( 2^{-(n_0-l)} \log\Big( 2\cdot \Big(\tfrac{4Md(c)}{1-2^{-\beta/2}}\Big)^2 \Big) + (n_0-l)\log\Big( 3^{\tfrac{2}{2\varrho}} \Big) \bigg) \Bigg] \Bigg)
\\
\label{eq:ThmSmallBallProbab5.3}
&\geq \exp\Big( -2^{n_0} c_3(d, T, M, \beta) \Big) \geq \exp\Big( -\tfrac{2c_3(d, T, M, \beta)}{\varepsilon} \Big).
\end{align}
Combining Equations \eqref{eq:ThmSmallBallProbab5.1}, \eqref{eq:ThmSmallBallProbab5.2} and \eqref{eq:ThmSmallBallProbab5.3} gives that
\begin{align*}
\eqref{eq:ThmSmallBallProbab1.1} \geq \exp\Big( -\tfrac{(c_1+c_2+c_3)}{\varepsilon} \Big)
\quad \Rightarrow\quad 
-\log\Big( \bP\Big[ \| \rw\|_\alpha \leq \varepsilon^\beta\Big] \Big) \lesssim \varepsilon^{-1} .
\end{align*}
\end{proof}

\subsection{Proof of Theorem \ref{Thm:SmallBallProbab-Both}}

This first result is a canonical adaption of the proof found in \cite{stolz1996some}*{Theorem 1.4} with the differentiability requirements weakened. We emphasise this proof is not original and included only for completeness. 

\begin{proof}[Proof of Proposition \ref{pro:SmallBallProbab-Lower}]
By the Cielsielski isomorphism (see \cite{herrmann2013stochastic}), we have that there exists a constant $\tilde{C}>0$ such that
$$
\sup_{s, t\in[0,T]} \frac{|W_{s, t}|}{|t-s|^{\alpha}} \geq \tilde{C} \sup_{p\in \bN_0} \sup_{m=1, ..., 2^{p}} 2^{p(\alpha-1/2)} |W_{(p,m)}|
$$
where
$$
W_{(p,m)} = 2^{p/2} \Big( W_{\tfrac{m-1}{T\cdot 2^{p}}, \tfrac{2m-1}{T\cdot 2^{p+1}}} - W_{\tfrac{2m-1}{T\cdot 2^{p+1}}, \tfrac{m}{T\cdot 2^{p}}} \Big). 
$$
Then for $q\in \bN_0$ such that $\tfrac{h}{2} \leq \tfrac{1}{T\cdot 2^{q}}<h$ and $p>q$, 
\begin{align*}
\| W \|_{\alpha}>& \tilde{C} \cdot \sup_{m=1, ..., 2^p} 2^{p(\alpha-1/2)} |W_{(p,m)}|
\\
>&\tilde{C} \cdot 2^{p(\alpha-1/2)} \cdot \sup_{n=0, ..., 2^{p-q}} \frac{1}{2^{p-q}} \sum_{m=1}^{2^{p-q}} \big| W_{(p,n\cdot 2^{p-q} + m)} \big|. 
\end{align*}
Thus for some choice of $p>q$ and $n=0, ..., 2^{p-q}$, 
$$
\bP\Big[ \| W \|_\alpha < \varepsilon \Big] \leq \bP\bigg[ \frac{1}{2^{p-q}} \cdot \sum_{m=1}^{2^{p-q}} \big| W_{(p,n\cdot 2^{p-q} +m)} \big| < \frac{\varepsilon \cdot 2^{p(1/2 - \alpha)}}{\tilde{C}} \bigg]. 
$$
From Equation \eqref{eq:Thm:SmallBallProbab-Lower-2} and Equation \eqref{eq:Thm:SmallBallProbab-Lower-1}, 
\begin{align*}
\bE\Big[ \big| W_{(p, m)} \big|^2 \Big] =& 2^{p} \bigg( 4 \cdot \sigma^2\Big( \tfrac{1}{2^{p+1}} \Big) - \sigma^2\Big( \tfrac{1}{2^p} \Big) \bigg)
\\
\geq& 2^p \cdot \frac{4-C_2}{4} \sigma^2\Big( \tfrac{1}{2^p} \Big) \geq 2^{p(1 - 1/\varrho)} \cdot \frac{C_1(4-C_2)}{4}. 
\end{align*}
Renormalising the wavelets gives
\begin{align*}
&\bP\bigg[ \frac{1}{2^{p-q}} \cdot \sum_{m=1}^{2^{p-q}} \big| W_{(p,n\cdot 2^{p-q} +m)} \big| < \frac{\varepsilon \cdot 2^{p(1/2 - \alpha)}}{\tilde{C}} \bigg]
\\
&\leq
\bP\Bigg[ \frac{1}{2^{p-q}} \cdot \sum_{m=1}^{2^{p-q}} \frac{\big| W_{(p,n\cdot 2^{p-q} +m)} \big|}{\bE\Big[ \big| W_{(p,n\cdot 2^{p-q} +m)} \big|^2 \Big]^{1/2}} < \frac{\varepsilon}{\tilde{C}}\cdot \sqrt{\tfrac{4}{C_1(4-C_2)}} \cdot 2^{p(\tfrac{1}{2\varrho} - \alpha)} \Bigg]. 
\end{align*}
Now
\begin{align*}
\bE\Big[ W_{(p, m_1)} \cdot W_{(p, m_2)} \Big] =& \frac{-2^p}{2}\bigg( \sigma^2\Big( \tfrac{|m_1 - m_2 - 1|}{2^p} \Big) - 4 \sigma^2\Big( \tfrac{|m_1 - m_2 - 1/2|}{2^p} \Big) + 6\sigma^2\Big( \tfrac{|m_1 - m_2|}{2^p} \Big)
\\
& - 4\sigma^2\Big( \tfrac{|m_1 - m_2 + 1/2|}{2^p} \Big) + \sigma^2\Big( \tfrac{|m_1 - m_2 + 1|}{2^p} \Big) \bigg)
\end{align*}
By considering a Taylor expansion of the function
$$
f(x) =  \sigma^2\Big( \tfrac{|n - x|}{2^p} \Big) - 4 \sigma^2\Big( \tfrac{|n - x/2|}{2^p} \Big) + 6\sigma^2\Big( \tfrac{|n|}{2^p} \Big) - 4\sigma^2\Big( \tfrac{|n + x/2|}{2^p} \Big) + \sigma^2\Big( \tfrac{|n + x|}{2^p} \Big) 
$$
and using that $f(0) = f'(0) = f''(0) = 0$, we get that $\exists \xi\in[0,1]$ such that
\begin{align*}
&f(1) 
\\
&= \tfrac{1}{12\cdot (2^p)^3} \Bigg( \bigg( \nabla^3\Big[ \sigma^2\Big]\Big( \tfrac{|n-\xi/2|}{2^p} \Big) - \nabla^3\Big[ \sigma^2\Big]\Big( \tfrac{|n+\xi/2|}{2^p} \Big) \bigg) - 2\bigg( \nabla^3\Big[ \sigma^2\Big]\Big( \tfrac{|n-\xi|}{2^p} \Big) - \nabla^3\Big[ \sigma^2\Big]\Big( \tfrac{|n+\xi|}{2^p} \Big) \bigg) \Bigg). 
\end{align*}
Applying this representation with Equation \eqref{eq:Thm:SmallBallProbab-Lower-3} gives that
\begin{equation}
\label{eq:pro:SmallBallProbab-Lower:proof2}
\frac{\bE\Big[ W_{(p,m_1)} \cdot W_{(p, m_2)} \Big] }{\bE\Big[ \big| W_{(p,m_1)} \big|^2 \Big]^{1/2}\cdot \bE\Big[ \big| W_{(p,m_1)} \big|^2 \Big]^{1/2}} \leq \tfrac{C_1(4-C_2)C_3}{8} \Big( |m_1 - m_2| - 1 \Big)^{\tfrac{1}{\varrho} - 3}
\end{equation}
Taking $\varepsilon$ small enough so that there exists a $p>q$ such that
\begin{equation}
\label{eq:pro:SmallBallProbab-Lower:proof1}
\tfrac{\tilde{C}}{M_1} \cdot \sqrt{\tfrac{C_1(4-C_2)}{4}} \cdot 2^{(p+1)(\alpha - \tfrac{1}{2\varrho})} 
\leq 
\varepsilon 
\leq 
\tfrac{\tilde{C}}{M_1} \cdot \sqrt{\tfrac{C_1(4-C_2)}{4}} \cdot 2^{p(\alpha - \tfrac{1}{2\varrho})}
\end{equation}
where $M_1$ is chosen from Lemma \ref{lemma:Stolz_Lem} gives us
$$
\bP\Big[ \| W \|_\alpha < \varepsilon \Big] 
\leq
\bP\Bigg[ \frac{1}{2^{p-q}} \cdot \sum_{m=1}^{2^{p-q}} \frac{\big| W_{(p,n\cdot 2^{p-q} +m)} \big|}{\bE\Big[ \big| W_{(p,n\cdot 2^{p-q} +m)} \big|^2 \Big]^{1/2}} < \frac{1}{M_1} \Bigg]. 
$$
Thanks to Equation \eqref{eq:pro:SmallBallProbab-Lower:proof2}, we can apply Lemma \ref{lemma:Stolz_Lem} to get
$$
\bP\Big[ \| W \|_\alpha < \varepsilon \Big] 
\leq 
\exp\Big( \tfrac{-2^{(p-q)}}{M_2} \Big). 
$$
However, by Equation \eqref{eq:pro:SmallBallProbab-Lower:proof1}, we have
$$
2^{p-q} \leq T \cdot h \cdot \Big[ \tfrac{M_1}{\tilde{C}} \cdot \sqrt{\tfrac{4}{C_1(4-C_2)}} \cdot 2^{1/(2\varrho) - \alpha} \Big]^{\tfrac{1}{\alpha - 1/(2\varrho)}} \cdot \varepsilon^{\tfrac{1}{\alpha - \tfrac{1}{2\varrho}}}
$$
so that
$$
\log\Big( \bP\Big[ \| W \|_\alpha < \varepsilon \Big] \Big) \lesssim -\varepsilon^{\tfrac{1}{\tfrac{1}{2\varrho} - \alpha}}. 
$$
\end{proof}

\begin{proof}[Proof of Theorem \ref{Thm:SmallBallProbab-Both}]
Suppose that $\sigma^2$ satisfies Equation \eqref{eq:Thm:SmallBallProbab-Both}. By \cite{frizhairer2014}*{Theorem 10.9}, Assumption \ref{assumption:GaussianRegularity} is satisfied. Hence, Theorem \ref{Thm:SmallBallProbab} implies
$$
\fB(\varepsilon) \lesssim \varepsilon^{\tfrac{-1}{\tfrac{1}{2\varrho} - \alpha}}. 
$$
On the other hand, Equation \eqref{eq:Thm:SmallBallProbab-Both} implies Equation \eqref{eq:Thm:SmallBallProbab-Lower-1}. The additional assumption that $\tfrac{c_2 \cdot 2^{1/\varrho}}{c_1}<4$ implies Equation \eqref{eq:Thm:SmallBallProbab-Lower-2}. Under the final assumption of Equation \eqref{eq:Thm:SmallBallProbab-Lower-3}, the assumptions of Proposition \ref{pro:SmallBallProbab-Lower} are satisfied. Finally, using the identity $\| W \|_\alpha  \lesssim \| \rw \|_\alpha$ gives
$$
\fB(\varepsilon) \gtrsim \varepsilon^{\tfrac{-1}{\tfrac{1}{2\varrho} - \alpha}}. 
$$

\end{proof}

\subsection{Limitations and further progress}

Gaussian rough paths have been successfully studied when the regularity of the path $\alpha \in (\tfrac{1}{4}, \tfrac{1}{3}]$. However, we do not address this example in this paper. 

Those familiar with the Philip-Hall Lie basis will realise that we will additionally need to account for the SBPs of terms of the form
$$
\int_s^t \int_s^r W_{s,q} \otimes dV_q \otimes dU_r , \quad \int_s^t \int_s^r W_{s,q} \otimes dW_q \otimes dV_r. 
$$
where $W$, $V$ and $U$ are independent, identically distributed Gaussian processes. The first term can be address with another application of Proposition \ref{pro:Sidak}. 

Let $(E_1, \cH_1, \cL_1)$, $(E_2, \cH_2, \cL_2)$ and $(E_3, \cH_3, \cL_3)$ be abstract Wiener spaces. The authors were able to demonstrate that when $\varepsilon$ is chosen to be small, for any sequence $h_i \in \cH_1 \otimes \cH_2 \otimes \cH_3$ and any choice of Gaussian measure $\cL$ over $E_1 \oplus E_2 \oplus E_3$ with marginals $\cL_1$, $\cL_2$ and $\cL_3$ satisfies
$$
\cL\Big[ \bigcap_{i} \big\{ | h_i(\hat{\otimes})|<\varepsilon \big\} \Big] \geq \Big( \cL_1 \times \cL_2 \times \cL_3\Big)\Big[ \bigcap_i \big\{ | h_i(\hat{\otimes})|<\varepsilon \big\} \Big]
$$

At face value, this would suggest the SBPs of terms of the form $\int_s^t \int_s^r W_{s,q} \otimes dW_q \otimes dV_r$ should be lower bounded by SBPs of terms of the form $\int_s^t \int_s^r W_{s,q} \otimes dV_q \otimes dU_r $. However, a key requirement is that $\varepsilon$ is chosen smaller that the variance of the functionals $h_i(\hat{\otimes})$ and for $\varepsilon$ large enough the inequality flips. 

This is naturally justified by the observation that the intersection of a ball with a hyperbola (both with common centre) is convex when the radius of the ball is small, but for large radius the set is not convex (so that one cannot apply Equation \eqref{eq:GaussCorrConj}). 

\newpage
\section{Metric entropy of Cameron-Martin balls}
\label{section:MetricEntropy}

This problem was first studied in \cite{Kuelbs1993} for Gaussian measures. While the law of a Gaussian rough path has many of the properties that Gaussian measures are known for, it is not itself a Gaussian so this result is not immediate. 

\begin{definition}
Let $(E, d)$ be a metric space and let $K$ be a compact subset of $E$. We define the $d$-metric entropy of $K$ to be $\fH(\varepsilon, K):=\log( \fN(\varepsilon, K)) $ where
$$
\fN(\varepsilon, K):=\min\left\{ n\geq 1: \exists e_1, ..., e_n\in E, \bigcup_{j=1}^n \bB(e_j, \varepsilon) \supseteq K\right\}
$$
and $\bB(e_i, \varepsilon):=\{ e\in E: d(e, e_i)<\varepsilon\}$. 
\end{definition}

Given a Gaussian measure $\cL^W$ with RKHS $\cH$ and unit ball $\cK$, let us consider the set of rough paths
\begin{equation}
\label{eq:RoughRKHSBall}
\rk:=\Big\{ \rh=S_2[h]: h\in \cK \Big\} \subset G\Omega_\alpha(\bR^d). 
\end{equation}

We can easily show that this set is \emph{equicontinuous} as a path on $G^2(\bR^d)$ so by the Arzel\`a–Ascoli theorem, see for example \cite{friz2010multidimensional}*{Theorem 1.4}, it must be compact in the metric space $WG\Omega_\alpha(\bR^d)$. Hence $\fN_{d_\alpha} (\varepsilon, \rk)$ is finite. 

\begin{theorem}
\label{thm:SmallBallMetricEnt1}
Let $\tfrac{1}{3} <\alpha<\tfrac{1}{2\varrho}$. Let that $\cL^W$ be a Gaussian measure and $\forall s, t\in [0,T]$, let $\bE\Big[ \big| W_{s, t} \big|^2 \Big] = \sigma^2\Big( \big| t-s \big| \Big)$. 

Suppose that 
\begin{enumerate}
\item $\exists h>0$ and $c_1, c_2>0$ such that $\frac{c_2 \cdot 2^{\tfrac{1}{\varrho}}}{c_1}<4$ and $\forall \tau \in [0,h)$, $\sigma^2(\tau)$ is convex and 
$$
c_1 \cdot |\tau|^{\tfrac{1}{\varrho}} \leq \sigma^2 \Big( \tau \Big) \leq c_2 \cdot | \tau |^{\tfrac{1}{\varrho}}. 
$$
\item $\exists c_3>0$ such that $\forall \tau \in [0,h)$, 
$$
\Big| \nabla^3 \big[ \sigma^2 \big] \Big( \tau \Big) \Big| \leq c_3 \cdot \tau^{\tfrac{1}{\varrho}-3}. 
$$
\end{enumerate}
Then the metric entropy of the set $\rk$ with respect to the H\"older metric satisfies
$$
\fH_{d_\alpha}( \varepsilon, \rk) \approx \varepsilon^{\tfrac{-1}{\tfrac{1}{2}+\tfrac{1}{2\varrho} - \alpha}}. 
$$
\end{theorem}

\begin{remark}
The maps $S_2: C^{1-var}([0,T]; \bR^d) \to G\Omega_\alpha(\bR^d)$ and $W \mapsto \rw$ are known to be measurable but not continuous. Therefore it is reasonably remarkable that this mapping takes a compact set to a compact set and that the two sets have the same metric entropy. 
\end{remark}

\subsection{Proof of Theorem \ref{thm:SmallBallMetricEnt1}}

In order to prove this, we first prove the following auxiliary result. 
\begin{proposition}
\label{pro:MetricEnt1}
Let $\cL^W$ be a Gaussian measure with RKHS $\cH$ satisfying Assumption \ref{assumption:GaussianRegularity}. Then for any $\eta,\varepsilon>0$, 
\begin{equation}
\label{eq:MetricEnt1.1}
\fH_{d_\alpha} \Big( 2\varepsilon, \delta_\eta (\rk) \Big) \leq \tfrac{\eta^2}{2} - \log\Big( \cL^\rw\Big[ \bB_\alpha(\rId, \varepsilon)\Big]\Big),
\end{equation}
and 
\begin{equation}
\label{eq:MetricEnt1.2}
\fH_{d_\alpha} \Big(\varepsilon, \delta_\eta (\rk) \Big)  \geq \log\Bigg( \Phi\bigg( \eta + \Phi^{-1}\Big( \cL^\rw\Big[ \bB_\alpha(\rId, \varepsilon)\Big]\Big)\bigg)\Bigg) - \log\Big( \cL^\rw\Big[\bB_\alpha (\rId, 2\varepsilon)\Big]\Big).
\end{equation}
\end{proposition}

\begin{proof}
Firstly, for some $\varepsilon>0$ consider the quantity
$$
\fM_{d_\alpha} (\varepsilon, \delta_\eta (\rk) ) = \max \Big\{ n\geq 1: \exists \rh_1, ..., \rh_n\in \delta_\eta(\rk) , d_\alpha( \rh_i, \rh_j)\geq 2\varepsilon \quad \forall i\neq j\Big\},
$$
and a set $\fF$ such that $| \fF|= \fM_{\delta_\alpha}(\varepsilon, \delta_\eta (\rk))$ and for any two distinct $\rh_1, \rh_2\in \fF$ that $d_\alpha(\rh_1, \rh_2)\geq 2\varepsilon$. Similarly, there must exist a set $\fG$ such that $| \fG| = \fN_{d_\alpha}(\varepsilon, \delta_\eta(\rk))$ and 
$$
\delta_\eta( \rk) \subseteq \bigcup_{\rh\in \fG} \bB_\alpha( \rh, \varepsilon). 
$$
Similarly, since $\fF$ is a maximal set, we also have
$$
\delta_\eta( \rk) \subseteq \bigcup_{\rh\in \fF} \bB_\alpha( \rh, 2\varepsilon),
$$
It is therefore natural that
$$
\fH\Big(2\varepsilon, \delta_\eta(\rk)\Big) \leq \log\Big( \fM_{d_\alpha}( \varepsilon, \delta_{\eta}(\rk))\Big)
\quad \mbox{ and } \quad
\fM_{d_\alpha} (\varepsilon, \delta_\eta (\rk) ) \min_{\rh \in \fF} \cL^\rw \Big[ \bB_\alpha ( \rh, \varepsilon) \Big] \leq 1. 
$$
By taking logarithms and applying Lemma \ref{lem:CamMartinFormRP} we get
$$
\log\Big( \fM_{d_\alpha} ( \varepsilon, \delta_\eta(\rk) ) \Big) - \tfrac{\eta^2}{2} + \log\Big( \cL^\rw\Big[ \bB_\alpha(\rId, \varepsilon)\Big]\Big) \leq 0,
$$
which implies \eqref{eq:MetricEnt1.1}. 

Secondly, from the definition of $\fF$ we get that
$$
T \Big( \bB_\alpha(\rId, \varepsilon), \delta_\eta(\rk)\Big) \subseteq \bigcup_{\rh\in \fG} \bB_\alpha(\rh, 2\varepsilon),
$$
where 
\begin{equation}
\label{eq:Ball+shift}
T \Big( \bB_\alpha(\rId, \varepsilon), \delta_\eta(\rk)\Big) := \Big\{ T^{\eta h}(\rx): \rx \in \bB_\alpha(\rId, \varepsilon), h\in \cK\Big\}. 
\end{equation}
Hence applying Lemma \ref{lem:AndersonInequalityRP} and Lemma \ref{lem:BorellInequalRP} gives
\begin{align*}
\fN_{d_\alpha} (\varepsilon, \delta_\eta(\rk)) \cdot \cL^\rw\Big[ \bB_\alpha(\rId, 2\varepsilon)\Big] 
\geq& \fN_{d_\alpha} (\varepsilon, \delta_\eta(\rk)) \cdot \max_{\rh\in \fG} \cL^\rw\Big[ \bB_\alpha( \rh, \varepsilon)\Big], 
\\
\geq& \cL^\rw \Big[ T \Big( \bB_\alpha(\rId, \varepsilon), \delta_\eta(\rk)\Big) \Big]  \geq \Phi\bigg( \eta + \Phi^{-1}\bigg( \cL^\rw\Big[ \bB_\alpha(\rId, \varepsilon)\Big]\bigg) \bigg),
\end{align*}
and taking logarithms yields \eqref{eq:MetricEnt1.2}. 
\end{proof}

\begin{proof}[Proof of Theorem \ref{thm:SmallBallMetricEnt1}]

Using Equation \ref{eq:MetricEnt1.1} for $\eta,\varepsilon>0$ we have
$$
\fH_{d_\alpha}( 2\varepsilon, \delta_\eta( \rk) ) \leq \tfrac{\eta^2}{2} + \fB(\varepsilon). 
$$
By the properties of the dilation operator it follows that $\fH_{d_\alpha}(\varepsilon, \delta_\eta(\rk)) = \fH_{d_\alpha}(\varepsilon/\eta, \rk)$. Making the substitution $\eta = \sqrt{ 2\fB(\varepsilon)}$ and using that $\fB$ is regularly varying at infinity leads to
$$
\fH_{d_\alpha} \Big( \tfrac{2\varepsilon}{\sqrt{2 \fB(\varepsilon)}}, \rk \Big) \leq 2 \fB(\varepsilon). 
$$
Finally, relabeling $\varepsilon' = \tfrac{2 \varepsilon}{\sqrt{2\fB(\varepsilon)}}$ which means $\varepsilon' \approx \sqrt{2} \varepsilon^{\frac{\beta+1/2}{\beta} }$ and $\beta =\tfrac{1}{2\varrho}-\alpha$, we apply Theorem \ref{Thm:SmallBallProbab} to obtain
$$
\fH_{d_\alpha} \Big( \varepsilon', \rk \Big) \leq \Big(\tfrac{2\varepsilon}{\varepsilon'}\Big)^2 \lesssim \varepsilon^{\frac{-1}{\tfrac{1}{2} +\beta} }. 
$$

For the second inequality, for $\eta,\varepsilon>0$, we use Equation \eqref{eq:MetricEnt1.2} with the substitution 
$$
-\eta = \Phi^{-1}\big( \cL^\rw[ \bB_\alpha(\rId, \varepsilon)]\big).
$$
This yields
$$
\fH_{d_\alpha}\Big( \tfrac{\varepsilon}{\Phi^{-1}\big( \cL^\rw[ \bB_\alpha(\rId, \varepsilon)]\big)}, \rk\Big) \geq \fB(2\varepsilon) + \log(1/2). 
$$

Next, using the known limit
$$
\lim_{x\to +\infty} \frac{-\Phi^{-1}(\exp(-x^2/2))}{x} = 1, 
$$
we equivalently have that
$$
\frac{\Phi^{-1}\Big(\exp\big(-\fB(\varepsilon)\big)\Big)^2 }{2} \sim \fB(\varepsilon),
$$
as $\varepsilon\to0$ since $\fB(\varepsilon)\to 0$. From here we conclude that as $\varepsilon\searrow  0$ we have 
$$
\frac{\varepsilon}{\Phi^{-1}\big( \cL^\rw[ \bB_\alpha(\rId, \varepsilon)]\big)} \sim \frac{\varepsilon}{\sqrt{2\fB(\varepsilon)} }. 
$$
Therefore, for $\varepsilon$ small enough and using that $\fB$ varies regularly, we obtain
$$
\fH_{d_\alpha}\Big( \tfrac{\varepsilon}{\sqrt{2 \fB(\varepsilon)}}, \rk \Big) \gtrsim \fB(2\varepsilon) + \log(1/2) \gtrsim \frac{\fB( \varepsilon)}{2} . 
$$
We conclude by making the substitution $\varepsilon' = \tfrac{\varepsilon}{\sqrt{2\fB(\varepsilon)}}$ and apply Theorem \ref{Thm:SmallBallProbab-Both}. 
\end{proof}

\newpage
\section{Optimal quantization and empirical distributions}
\label{section:OptimalQuant}

In this section, we prove the link between metric entropy and optimal quantization and solve the asymptotic rate of convergence for the quantization problem of a Gaussian rough path. 

\subsection{Introduction to finite support measures}

For a neat introduction to Quantization, see \cite{graf2007foundations}.

\begin{definition}
Let $(E, d)$ be a separable metric space endowed with the Borel $\sigma$-algebra $\cB$. For $r\geq1$, we denote $\cP_r(E)$ to be the space of integrable measures over the measure space $(E, \cB)$ with finite $r^{th}$ moments. For $\mu, \nu\in \cP_r(E)$, we denote the Wasserstein distance $\bW_{d}^{(r)}: \cP_r(E) \times \cP_r(E) \to \bR^+$ to be
$$
\bW^{(r)}_{d}(\mu, \nu) = \inf_{\gamma \in \cP(E \times E)} \bigg( \int_{E \times E} d(x, y)^r \gamma(dx, dy) \bigg)^{\tfrac{1}{r}}
$$
where $\gamma$ is a joint distribution over $E\times E$ which has marginals $\mu$ and $\nu$. 
\end{definition}

The Wasserstein distance induces the topology of weak convergence of measure as well as convergence in moments of order up to and including $2$.

\begin{definition}
Let $I$ be a countable index, let $\fS:=\{ \fs_i, i\in I\}$ be a partition of $E$ and let $\fC:=\{\fc_i\in E, i\in I\}$ be a codebook. 

Define $\fQ$ be the set of all quantizations $q:E \to E$ such that 
\begin{align*}
q(x) = \fc_i &\quad \mbox{for $x\in \fs_i$}, \qquad q(E) = \fC
\end{align*}
for any possible $\fS$ and $\fC$. Then the pushforward measure by the function $q$ is
$$
\cL \circ q^{-1}(\cdot) = \sum_{i\in I} \cL(\fs_i) \delta_{\fc_i} (\cdot) \in \cP_2(E). 
$$
\end{definition}


\begin{definition}[Optimal quantizers]
\label{dfn:OptimalQuantizer}
Let $n\in \bN$ and $r\in[1, \infty)$. The minimal $n^{th}$ quantization error of order $r$ of a measure $\cL$ on a separable metric space $E$ is defined to be
$$
\fE_{n, r}(\cL) = \inf\Bigg\{ \Big( \int_E \min_{\fc\in \fC} d( x, \fc)^r d\cL(x) \Big)^{\tfrac{1}{r}}: \fC\subset E, 1\leq |\fC| \leq n\Bigg\}. 
$$

A codebook $\fC = \{\fc_i, i\in I\}$ with $1\leq|\fC|\leq n$ is called an $n$-optimal set of centres of $\cL$ (of order r) if
$$
\fE_{n, r}(\cL) = \Big( \int_E \min_{i=1, ..., n} d(x, \fc_i)^r d\cL(x) \Big)^{\tfrac{1}{r}}
$$
\end{definition}

\begin{remark}
Suppose that one has found an $n$-optimal set of centres for a measure $\cL$. Then an optimal partition can be obtained for any collection of sets $\fs_i$ such that
$$
\fs_i \subset \{ x \in E: d(x, \fc_i) \leq d(x, \fc_j), j=1, ..., n\}, 
\quad
\fs_i \cap \fs_j = \emptyset, i\neq j
$$
and $\cup_{i=1}^n \fs_i = E$. 

Frequently, this means that the partition is taken to be the interior of the collection of Voronoi sets with centres equal to the $n$-optimal codebook plus additions that account for the boundaries between the Voronoi sets. 
\end{remark}

\subsection{Optimal Quantization}

This next result follows the ideals of \cite{graf2003functional}, although a similar result proved using a different method can be found in \cite{dereich2003link}. 

\begin{theorem}
\label{thm:QuantizationRateCon}
Let $\cL^W$ be a Gaussian measure satisfying Assumption \ref{assumption:GaussianRegularity} and let $\cL^\rw$ be the law of the lift to the Gaussian rough path. Then for any $1\leq r <\infty$
\begin{equation}
\label{eq:QuantizationRateCon}
\fB^{-1}\Big( \log(2n)\Big) \lesssim \fE_{n, r}(\cL^\rw) 
\end{equation}
where $\fB$ is the SBP of the measure $\cL^\rw$. 
\end{theorem}

In particular, if we additionally have that
$$
\lim_{\varepsilon \to 0} \frac{\cL^\rw[ \bB(\rId, \tfrac{\varepsilon}{2})]}{\cL^\rw[ \bB(\rId, \varepsilon)]} = 0
$$
then
$$
\fB^{-1}\Big( \log(n)\Big) \lesssim \fE_{n, r}(\cL^\rw). 
$$

\begin{proof}
Let the set $\fC_{n, r} \subset G\Omega_\alpha(\bR^d)$ be a codebook containing $n$ elements . We know that the function $\min_{\fc \in \fC_n} d_{\alpha}(\rx, \fc)^r$ will be small in the vicinity of $\fC_{n, r}$, so we focus on when it is large. Thus
\begin{align*}
\int& \min_{\fc \in \fC_n} d_{\alpha}(\rx, \fc)^r d\cL^\rw(\rx)
\geq 
\int_{\Big( \bigcup_{\fc\in \fC_{n,r}} \bB_\alpha \Big(\fc, \fB^{-1}(\log(2n))\Big) \Big)^c} \min_{\fc \in \fC_n} d_{\alpha}(\rx, \fc)^r d\cL^\rw(\rx), 
\\
\geq& \fB^{-1}\Big(\log(2n)\Big)^r \left( 1- \cL^\rw\Big[  \bigcup_{\fc\in \fC_{n, r}} \bB_\alpha \Big(\fc, \fB^{-1}(\log(2n)) \Big) \Big] \right), 
\\
\geq& \fB^{-1}\Big(\log(2n)\Big)^r \left( 1- n\cL^\rw\Big[ \bB_\alpha \Big(\rId,  \fB^{-1}(\log(2n)) \Big) \Big] \right)
\geq \frac{\fB^{-1}\Big(\log(2n)\Big)^r}{2}
\end{align*}
by applying Lemma \ref{lem:AndersonInequalityRP}. Now taking a minimum over all possible choices of codebooks, we get
$$
\fE_{n, r}(\cL^\rw)^r \gtrsim \fB^{-1}\Big(\log(2n)\Big)^r. 
$$
\end{proof}

\subsection{Convergence of weighted empirical measure}

We now turn our attention to the problem of sampling and the rate of convergence of empirical measures. In general, the quantization problem is only theoretical as obtaining the codebook and partition that attain the minimal quantization error is computationally more complex than beneficial. An empirical distribution removes this challenge at the sacrifice of optimality and the low probability event that the approximation will be far in the Wasserstein distance from the true distribution. 

\begin{definition}
\label{dfn:WeightedEmpiricalMeasure}
For an enhanced Gaussian measure $\cL^\rw$, let $(\Omega, \cF, \bP)$ be a probability space containing $n$ independent, identically distributed enhanced Gaussian processes $(\rw^i)_{i=1, ..., n}$. Let $\fs_i$ be a Voronoi partition of $WG\Omega_{\alpha}(\bR^d)$, 
$$
\fs_i \subset \Big\{ \rx\in WG\Omega_{\alpha}(\bR^d): d_\alpha( \rx, \rw^i) = \min_{j=1, ..., n} d_\alpha( \rx, \rw^i)\Big\}, 
\quad
\fs_i \cap \fs_j=\emptyset
$$
and $\bigcup_{i=1}^n \fs_i = WG\Omega_\alpha (\bR^d)$. 

Then we define the \emph{weighted empirical measure} to be the random variable $\cM: \Omega \to \cP_2\big(  WG\Omega_{\alpha}(\bR^d)\big)$
\begin{equation}
\label{eq:dfn:WeightedEmpiricalMeasure}
\cM_n = \sum_{i=1}^n \cL^\rw(\fs_i) \delta_{\rw^i}. 
\end{equation}
\end{definition}

Note that the quantities $\cL^\rw(\fs_i)$ are random and $\sum_{i=1}^n \cL^\rw(\fs_i)=1$. The weights are in general NOT uniform. We think of $\cM_n$ as a (random) approximation of the measure $\cL^\rw$ and in this section we study the random variable $\bW_{d_\alpha}^{(2)} (\cM_n, \cL^\rw)$ and its mean square convergence to 0 as $n\to \infty$. 

This next Theorem is an adaption of the method found in \cite{dereich2003link}. 

\begin{theorem}
\label{thm:EmpiricalRateCon1}
Let $\cL^W$ be a Gaussian measure satisfying Assumption \ref{assumption:GaussianRegularity} and let $\cL^\rw$ be the law of the lift to the Gaussian rough path. Let $(\Omega, \cF, \bP)$ be a probability space containing a sequence of independent, identically distributed Gaussian rough paths with law $\cL^\rw$. Let $\cM_n$ be the empirical measure with samples drawn from the measure $\cL^\rw$. Then for any $1\leq r <\infty$
\begin{equation}
\label{eq:thm:EmpiricalRateCon1}
\bE\left[ \bW_{d_\alpha}^{(r)}\Big(\cL^\rw, \cM_n\Big)^r \right]^{1/r} \lesssim \fB^{-1}\Big( \log(n)\Big). 
\end{equation}
where $\fB$ is the SBP of the measure $\cL^\rw$. 
\end{theorem}

\begin{proof}
By the definition of the Wasserstein distance, we have
\begin{align*}
\bW_{d_\alpha}^{(r)}\Big( \cL^\rw, \cM_n\Big)^r =& \inf_{\gamma \in \cP\big(WG\Omega_\alpha(\bR^d) \times WG\Omega_\alpha(\bR^d)\big)} \int_{WG\Omega_\alpha(\bR^d) \times WG\Omega_\alpha(\bR^d)} d_\alpha( \rx, \ry)^r \gamma(d\rx, d\ry)
\\
\leq& \int_{WG\Omega_\alpha(\bR^d)} \min_{j=1, ..., n} d_\alpha\Big(\rx, \rw^j\Big)^r d\cL^{\rw}(\rx)
\end{align*}

Thus taking expectations, we have
\begin{align*}
&\bE\left[ \bW_{d_\alpha}^{(r)} \Big( \cL^\rw, \cM_n\Big)^r \right] 
\\
&\leq \int_{WG\Omega_{\alpha}(\bR^d)^{\times n}} \left( \int_{WG\Omega_{\alpha}(\bR^d)} \min_{j=1, ..., n} d_\alpha\Big( \rx, \rw^j\Big)^r d\cL^\rw(\rx) \right)d\left( \cL^\rw\right)^{\times n} (\rw^1, ..., \rw^n)
\end{align*}

A change in the order of integration yields
\begin{equation}
\label{eq:EmpiricalRateCon1.1}
\bE\left[ \bW_{d_\alpha}^{(r)} \Big( \cL^\rw, \cM_n\Big)^r \right] 
\leq 
2^r \int_0^\infty \int_{WG\Omega_{\alpha}(\bR^d)} \left( 1-\cL^\rw\Big[ \bB_\alpha( \rx, 2\varepsilon^{1/r})\Big] \right)^n d\cL^\rw(\rx) d\varepsilon. 
\end{equation}

Firstly, choose $n$ large enough so that for any $0<\varepsilon<1$, $\sqrt{ \log(n) } > \Phi^{-1}\Big( \cL^\rw \Big[ \bB_\alpha(1, \varepsilon^{1/r})\Big] \Big)$. Secondly, choose $c>0$ such that $\cL^\rw\Big[ \bB_\alpha(1, c^{1/r}) \Big] \leq \Phi\Big(\tfrac{-1}{\sqrt{2\pi}} \Big)$. 

For $n$ and $\varepsilon$ fixed, we label the set
$$
A_{\varepsilon, n}:= \left\{ \rx\in WG\Omega_{\alpha}(\bR^d): I(\rx, \varepsilon) \leq \frac{\left( \tfrac{\sqrt{\log(n)}}{3} - \Phi^{-1}\Big( \cL^\rw\big[ \bB_\alpha(1, \varepsilon^{1/r})\big] \Big)\right)^2}{2} \right\}
$$
where $I(\rx, \varepsilon)$ was introduced in Definition \ref{definition:Freidlin-Wentzell_Function}. 

This can equivalently be written as
$$
A_{\varepsilon, n}:= \left\{ T^h[\rx]\in WG\Omega_{\alpha}(\bR^d): h\in \left( \tfrac{\sqrt{\log(n)}}{3} - \Phi^{-1} \Big( \cL^\rw\Big[ \bB_\alpha(\rId, \varepsilon^{1/r}) \Big]\Big)\right)\cK, \quad \rx\in \bB_\alpha(\rId, \varepsilon) \right\}
$$

Then we divide the integral in Equation \ref{eq:EmpiricalRateCon1.1} into
\begin{align}
\label{eq:EmpiricalRateCon2.1}
\bE\left[ \bW_{d_\alpha}^{(r)} \Big( \cL^\rw, \cM_n\Big)^r \right] 
\leq& 
2^r \int_0^c \int_{A_{\varepsilon, n}} \left( 1-\cL^\rw\Big[ \bB_\alpha( \rx, 2\varepsilon^{1/r})\Big] \right)^n d\cL^\rw(\rx) d\varepsilon
\\
\label{eq:EmpiricalRateCon2.2}
&+2^r \int_0^c \int_{A_{\varepsilon, n}^c} \left( 1-\cL^\rw\Big[ \bB_\alpha( \rx, 2\varepsilon^{1/r})\Big] \right)^n d\cL^\rw(\rx) d\varepsilon
\\
\label{eq:EmpiricalRateCon2.3}
&+2^r \int_c^\infty \int_{ WG\Omega_{\alpha}(\bR^d) } \left( 1-\cL^\rw\Big[ \bB_\alpha( \rx, 2\varepsilon^{1/r})\Big] \right)^n d\cL^\rw(\rx) d\varepsilon. 
\end{align}

Firstly, using Corollary \ref{cor:CamMartinFormRP}
\begin{align*}
\eqref{eq:EmpiricalRateCon2.1} \leq& 2^r\int_0^c \int_{A_{\varepsilon, n}} \left( 1- \exp\Big( - I(\rx, \varepsilon) - \fB(\varepsilon^{1/r})\Big) \right) d\cL^\rw (\rx) d\varepsilon
\\
\leq& 2^r \int_0^c \left( 1- \exp\left( -\Big( \tfrac{\sqrt{\log(n)}}{3} + \sqrt{2 \fB(\varepsilon^{1/r}) } \Big)^2 \right) \right)^n d\varepsilon
\\
\leq& 2^r\int_0^{\fB^{-1}\Big(\tfrac{\log(n)}{8}\Big)^r} d\varepsilon + 2^r \int_{\fB^{-1}\Big(\tfrac{\log(n)}{8}\Big)^r}^{ \fB^{-1}\Big(\tfrac{\log(n)}{8}\Big)^r \vee c} \left( 1- \exp\left( -\Big( \tfrac{\sqrt{\log(n)}}{3} + \sqrt{2 \fB(\varepsilon^{1/r}) } \Big)^2 \right)\right)^n d\varepsilon
\\
\leq& 2^r \fB^{-1}\Big(\tfrac{\log(n)}{8}\Big)^r + 2^r \left( 1- \exp\left( -\frac{25 \log(n)}{36} \right)\right)^n
\end{align*}
since $c<1$. 

Now, since $\log\Big( \tfrac{1}{\varepsilon}\Big) = o\Big( \fB(\varepsilon)\Big)$ as $\varepsilon \to 0$, we have that $\forall p, q$ that $\exp\Big( \tfrac{-n}{p}\Big) = o\Big( \fB^{-1}(n)^q \Big)$. Therefore
\begin{align*}
\left( 1- \exp\left( -\frac{25 \log(n)}{36} \right)\right)^n =& \Big( 1 - \frac{1}{n^{25/36} } \Big)^n
\\
\leq& \exp\Big(-n^{11/36}\Big) = o\left( \exp\Big( \frac{-\log(n)}{2}\Big) \right). 
\end{align*}

Next, applying Lemma \ref{lem:BorellInequalRP} to
\begin{align*}
\cL^\rw \Big[ A_{\varepsilon, n}^c \Big] =& 1 - \cL^\rw \Big[ A_{\varepsilon, N} \Big]
\\
\leq& 1 - \Phi\left( \Phi^{-1}\Big( \cL^\rw\Big[ \bB_\alpha(\rId, \varepsilon^{1/r}) \Big]\Big) + \frac{\sqrt{\log(n)}}{3} - \Phi^{-1}\Big( \cL^\rw\Big[ \bB_\alpha(\rId, \varepsilon^{1/r}) \Big]\Big) \right)
\\
=& 1 - \Phi\left( \frac{\sqrt{\log(n)}}{3} \right) \leq \exp\left( - \frac{\log(n)}{18} \right). 
\end{align*}

Third and finally, we make the substitution
\begin{align}
\nonumber
\eqref{eq:EmpiricalRateCon2.3} \leq& 2^r \int_{ WG\Omega_{\alpha}(\bR^d) } \int_0^\infty \cL^\rw\Big[ \bB^c( \rx, 2\varepsilon^{1/r})\Big] d\varepsilon \cdot  \cL^\rw\Big[ \bB^c( \rx, 2 c^{1/r})\Big] ^{n-1}  d\cL^\rw(\rx) 
\\
\label{eq:EmpiricalRateCon3.1}
\leq& \bE\Big[ \| \rw\|_\alpha^r \Big] \int_{ WG\Omega_{\alpha}(\bR^d) } \cL^\rw\Big[ \bB^c( \rx, 2 c^{1/r})\Big] ^{n-1}  d\cL^\rw(\rx)
\\
\label{eq:EmpiricalRateCon3.2}
&+ \int_{ WG\Omega_{\alpha}(\bR^d) } \| \rx \|_{\alpha}^r \cdot \cL^\rw\Big[ \bB^c( \rx, 2 c^{1/r})\Big] ^{n-1}  d\cL^\rw(\rx)
\end{align}
to account for the integral over $(c, \infty)$. Next, we partition this integral over $A_{c, n}$ and $A_{c, n}^c$. Arguing as before, we have
\begin{align*}
\int_{ A_{c, n} }& \left( 1 - \cL^\rw\Big[ \bB_\alpha( \rx, 2 c^{1/r})\Big] \right)^{n-1}  d\cL^\rw(\rx)
\\
&\leq \left( 1 - \exp\Big( - \frac{25 \log(n)}{36} \Big) \right)^{n-1} \leq o\Big( \fB^{-1} \Big(\log(n)\Big)^q \Big)
\end{align*}
and
\begin{align*}
\int_{ A_{c, n}^c }& \left( 1 - \cL^\rw\Big[ \bB_\alpha( \rx, 2 c^{1/r})\Big] \right)^{n-1}  d\cL^\rw(\rx)
\\
&\leq \sup_{0<c} \cL^\rw \Big[ A_{c, n}^c\Big] \leq o\Big( \fB^{-1}\Big( \log(n)\Big)^q \Big). 
\end{align*}
for any choice of $q$. Therefore
\begin{align*}
\eqref{eq:EmpiricalRateCon3.1} \leq& o\Big( \fB^{-1}\Big( \log(n)\Big)^q \Big)
\\
\eqref{eq:EmpiricalRateCon3.2} \leq& \bE\Big[ \| \rw \|_\alpha^{2r} \Big]^{1/2} \cdot \left( \int_{ WG\Omega_{\alpha}(\bR^d) } \cL^\rw \Big[ \bB_\alpha^c( \rx, 2c^{1/r}) \Big]^{2(n-1)} d\cL^\rw(\rx) \right)^{1/2}
\\
\leq& o\Big( \fB^{-1}\Big( \log(n)\Big)^q \Big)
\end{align*}
Thus
$$
\eqref{eq:EmpiricalRateCon1.1} \leq 2^r \fB^{-1}\Big(\tfrac{\log(n)}{8}\Big)^r +  o\Big( \fB^{-1}\Big( \log(n)\Big)^r \Big). 
$$
\end{proof}

\subsection{Convergence of (non-weighted) empirical measure}

In this Section, we study the empirical measure with uniform weights. This is the form that the empirical distribution more commonly takes. 

\begin{definition}
\label{dfn:EmpiricalMeasure}
For enhanced Gaussian measure $\cL^\rw$, let $(\Omega, \cF, \bP)$ be a probability space containing $n$ independent, identically ditributed enhanced Gaussian processes $(\rw^i)_{i=1, ..., n}$. 

Then we define the empirical measure to be the random variable $\cE_n: \Omega \to \cP_2\big(  WG\Omega_{\alpha}(\bR^d)\big)$
\begin{equation}
\label{eq:dfn:EmpiricalMeasure}
\cE_n = \frac{1}{n} \sum_{i=1}^n \delta_{\rw^i}. 
\end{equation}
\end{definition}

Before, we state our Theorem for the rate of convergence of empirical measures, we will need the following Lemma. 

\begin{lemma}
\label{lemma:Boissard}
Let $\eta>0$. Let $(E, d)$ be a metric space and let $\mu \in \cP(E)$ with $\supp(\mu) \subset E$ such that $\fN(t, \supp(\mu)) < \infty$ and let 
$$
\Delta_\mu:= \max \Big\{ d(x, y): x, y \in \supp(\mu) \Big\}. 
$$
Then $\exists c>0$ such that
$$
\bE\Big[ \bW_{d_\alpha}^{(r)}( \cE_n, \mu) \Big] \leq c \bigg( \eta + n^{-1/2r} \int_\eta^{\Delta_\mu/4} \fN\Big( z, \supp(\mu) \Big) dz \bigg). 
$$
\end{lemma}

For a proof of Lemma \ref{lemma:Boissard}, see \cite{boissard2014mean}. 

\begin{theorem}
\label{thm:EmpiricalRateCon2}
Let $\cL^W$ be a Gaussian measure satisfying Assumption \ref{assumption:GaussianRegularity} and let $\cL^\rw$ be the law of the lift to the Gaussian rough path. Let $(\Omega, \cF, \bP)$ be a probability space containing a sequence of independent, identically distributed Gaussian rough paths with law $\cL^\rw$. Let $\cE_n$ be the empirical measure with samples drawn from the measure $\cL^\rw$. Then for any $1\leq r <\infty$
\begin{equation}
\label{eq:thm:EmpiricalRateCon2}
\bE\bigg[ \bW_{d_\alpha}^{(r)}\Big(\cL^\rw, \cE_n\Big)^r \bigg] 
\lesssim
\fB^{-1}\Big( \log(n)\Big). 
\end{equation}
\end{theorem}

Typically, the measure $\cE_n$ is easier to work with since one does not have to calculate the (random) weights associated to each sample. However, it is a less accurate approximation than the weighted empirical so the fact that it converges at the same rate is noteworthy. 

\begin{proof}
Using the same construction as in Equation \eqref{eq:Ball+shift}, we consider the set
$$
\rS:= T\Big( \bB_\alpha(\rId, \varepsilon), \delta_\lambda(\rk) \Big)
$$
and denote the conditional measure $\tilde{\cL}^\rw = \tfrac{1}{\cL^\rw[\rS]} \1_{\rS} \cL^\rw$. An application of Lemma \ref{lemma:Boissard} gives
\begin{align}
\nonumber
\bE\bigg[ \bW_{d_\alpha}^{(r)}\Big(\cL^\rw, \cE_n\Big)^r \bigg]^{1/r}
\leq&
\bE\bigg[ \bW_{d_\alpha}^{(r)}\Big(\cL^\rw, \tilde{\cL}^\rw \Big)^r \bigg]^{1/r} + 
\bE\bigg[ \bW_{d_\alpha}^{(r)}\Big(\cE_n, \tilde{\cL}^\rw \Big)^r \bigg]^{1/r}
\\
\label{eq:thm:EmpiricalRateCon2-1}
\leq& \bE\bigg[ \bW_{d_\alpha}^{(r)}\Big(\cL^\rw, \tilde{\cL}^\rw \Big)^r \bigg]^{1/r}
+
2c\varepsilon + cn^{-1/4} \int_{2\varepsilon}^{\tfrac{\sigma\lambda + \varepsilon}{2}} \fN(z, \rS)^{1/4} dz
\end{align}
where
$$
\sigma: = \sup_{s, t \in [0,T]} \frac{\bE\Big[ |W_{s, t}|^2 \Big]^{1/2}}{|t-s|^\alpha}. 
$$
An application of the Talagrand inequality for rough paths (see for example \cite{riedel2017transportation} and Lemma \ref{lem:BorellInequalRP} gives that
$$
\bE\bigg[ \bW_{d_\alpha}^{(r)}\Big(\cL^\rw, \tilde{\cL}^\rw \Big)^r \bigg]^{1/r} \leq \sqrt{2\sigma^2} \sqrt{-\log\bigg( \Phi\Big( \lambda + \Phi^{-1}(e^{-\fB(\varepsilon)})\Big) \bigg) }. 
$$
Next, using the same estimates for $\Phi$ as in the proof of Theorem \ref{thm:SmallBallMetricEnt1}, we get that
$$
\bE\bigg[ \bW_{d_\alpha}^{(r)}\Big(\cL^\rw, \tilde{\cL}^\rw \Big)^r \bigg]^{1/r} \leq 
2\sigma \exp \bigg( \tfrac{-1}{4} \Big( \lambda - \sqrt{2 \fB(\varepsilon)} \Big)^2 \bigg)
$$
provided $\fB(\varepsilon)\geq \log(2)$ and $\lambda \geq \sqrt{2 \fB(\varepsilon)}$. 

Secondly, 
\begin{align*}
\frac{1}{n^{1/4}} \int_{2\varepsilon}^{\tfrac{\sigma\lambda + \varepsilon}{2}} \fN(z, \rS)^{1/4} dz 
\leq& 
\frac{\sigma\lambda + \varepsilon}{2n^{1/4}} \cdot \fN(2\varepsilon, \rS)^{1/4} \cdot 
\leq \frac{\sigma \lambda + \varepsilon}{2 n^{1/4}} \cdot \fN( \varepsilon , \lambda \rk). 
\end{align*}

By Proposition \ref{pro:MetricEnt1}, we have that
$$
\fN(\varepsilon, \lambda \rk) \leq \exp\Big( \tfrac{\lambda^2}{2} + \fB\big(\tfrac{\varepsilon}{2} \big) \Big) \leq \exp\Big( \tfrac{\lambda^2}{2} + \kappa\cdot \fB(\varepsilon) \Big). 
$$
Now we choose
$$
\varepsilon = \fB^{-1}\Big( \tfrac{1}{6+\kappa} \cdot \log(n)\Big), 
\quad 
\lambda = 2\sqrt{\tfrac{2}{6+\kappa}\cdot \log(n)}
$$
and Equation \eqref{eq:thm:EmpiricalRateCon2-1} becomes
\begin{align*}
\bE\bigg[ \bW_{d_\alpha}^{(r)}\Big(\cL^\rw, \cE_n\Big)^r \bigg]^{1/r}
\leq
c\bigg( \fB^{-1}\Big( \tfrac{1}{6+\kappa} \cdot \log(n)\Big) + \Big( 1 + \sigma\sqrt{\tfrac{1}{6+\kappa} \log(n)} \Big) n^{\tfrac{-1}{12+2\kappa}} \bigg)
\end{align*}
provided $n$ is large enough. As $n\to \infty$, the lead term will be $\fB^{-1}\Big(\log(n)\Big)$. 
\end{proof}

\subsection*{Acknowledgements}

The author is grateful to Professor Dr Peter Friz and Dr Yvain Bruned whose examination of the author's PhD thesis lead to several improvements in the mathematics on which this paper is based. 

The author is also very grateful to Dr Gon\c calo dos Reis who supervised the PhD and allowed the author to pursue the author's own projects. 

The author is grateful to the anonymous reviewer whose feedback lead to a number of improvements to the paper. 

The author acknowledges with thanks the funding received from the University of Edinburgh during the course of the author's PhD. 

\bibliographystyle{alpha}


\end{document}